\documentclass[leqno,11pt]{amsart}

\usepackage{amsmath,amsfonts,amsthm,mathrsfs,amssymb}
\usepackage[all]{xy}
\usepackage{enumerate}
\usepackage{graphicx}
\usepackage{bm}
\usepackage{epsfig}
\usepackage{textcomp}
\usepackage{multicol}

\usepackage[lmargin=3.2cm,rmargin=3.2cm,tmargin=4cm,bmargin=4.5cm]{geometry}
\usepackage{setspace}

\theoremstyle{plain}
\newtheorem{thm}{Theorem}[section]
\newtheorem{cor}[thm]{Corollary}
\newtheorem{lem}[thm]{Lemma}

\newtheorem{prop}[thm]{Proposition}

\theoremstyle{definition}
\newtheorem{defn}[thm]{Definition}
\newtheorem{exmp}[thm]{Example}
\newtheorem{rem}[thm]{Remark}

\newtheorem{ques}[thm]{Question}
\newtheorem{nota}[thm]{Notation}

\newcommand\lto{\longrightarrow}
\newcommand\nto{\stackrel}

\def\ZZ{{\bf Z}}

\def\PP{{\bf P}}
\def\RR{{\bf R}}
\def\CC{{\bf C}}

\newcommand\Cc{\mathcal{C}}

\def\ker{\mathrm{ker}}

\newcommand\coker{\mathrm{coker}}

\def\deg{\mathrm{deg}}

\newcommand\ann{\mathrm{ann}}

\newcommand\Hom{\mathrm{Hom}}

\def\dim{\mathrm{dim}}

\newcommand\Ext{\mathrm{Ext}}

\newcommand\tor{\textnormal{Tor}}

\newcommand\X{\textbf{X}}

\def\h{\textbf{h}}

\newcommand\Supp{\mathrm{Supp}}

\def\kos{\mathbb{K}}
\def\k.{\mathbb{K}_{\bullet}}
\def\h.{\mathbb{H}_{\bullet}}
\def\HHH{\mathbb{H}}

\def\FF{\mathcal F}

\def\CCC{\mathscr C}

\def\EE{\mathcal E}

\def\SS{\mathcal S}

\def\EEE{\mathscr E}

\def\d{\delta}

\def\pp{\mathfrak{p}}
\def\mm{\mathfrak{m}}

\def\Ss{\mathfrak{S}}

\def\pp{\mathfrak{p}}

\def\mm{\mathfrak{m}}

\def\P1{\PP^1}

\def\C.{C_\bullet}

\def\G{G}

\newcommand\cd{\textnormal{cd}}

\def\F.{F_\bullet}

\newcommand\reg{\textnormal{reg}}

\newcommand\ra{\rightarrow}
\newcommand\fin{\textnormal{end}}

\newcommand\colimit[1]{\underset{#1}{\underset{\lto}{\text{lim}}} \ }

\title{{Castelnuovo Mumford Regularity with respect to multigraded ideals}}
 
\author[Nicol\'{a}s Botbol]{Nicol\'{a}s Botbol} \thanks{NB was partially supported by UBACYT 20020100100242, CONICET PIP 112-200801-00483, and ANPCyT PICT 2008-0902, CONICET PhD and Postdoctoral Fellowships, Argentina; Marie-Curie Network ``SAGA" FP7 contract PITN-GA-2008-214584, EU}
\address{Departamento de Matem\'atica, FCEN, Universidad de Buenos Aires, Argentina }
\email{nbotbol@dm.uba.ar}
\urladdr{http://mate.dm.uba.ar/~nbotbol/}

\author[Marc Chardin]{Marc Chardin}
\address{Institut de Math\'ematiques de Jussieu.
UPMC, Boite 247, 4, place Jussieu,
F-75252 PARIS CEDEX 05}
\email{chardin@math.jussieu.fr}
\urladdr{http://people.math.jussieu.fr/~chardin/}

\date{\today}

\begin{document}

\begin{abstract}
In this article we extend a previous definition of Castelnuovo-Mumford regularity for modules over an algebra graded by a finitely generated abelian group.

Our notion of regularity is based on Maclagan and Smith's definition, and is extended first by working over any commutative base ring, and 
second by considering local cohomology with support in an arbitrary finitely generated graded ideal $B$, obtaining, for each $B$, a $B$-regularity region. The first extension provides a natural approach for working with families of sheaves or of graded modules, while the second opens new applications. 

We provide tools to transfer knowledge in two directions. First to deduce some information on the graded Betti numbers from the knowledge of regions where the local cohomology with support in a given graded ideal vanishes. This is one of our main results. Conversely,  vanishing of local cohomology with support in any graded ideal is deduced from the shifts in a free resolution and the local cohomology of the polynomial ring. Furthermore, the flexibility of treating local cohomology with respect to any $B$ open new possibilities for passing information.

We provide new persistence results for the vanishing of local cohomology that extend the fact that weakly regular implies regular in the classical case, and we give sharp estimates for the regularity of a truncation of  a module.

In the last part, we present a result on Hilbert functions for multigraded polynomial rings, that in particular provides a simple proof of the Grothendieck-Serre formula.
\end{abstract}

\maketitle

\noindent{\footnotesize Math.\ Sub.\ Class.: 13D45, 13D02, 13D07, 14M25, 14B15.}
\medskip 


\section{Introduction.}\label{sec:introCMReg}

Castelnuovo-Mumford regularity is a fundamental invariant in commutative algebra and algebraic geometry. In the classical case
(standard graded algebras) it measures the maximum degree of the syzygies and provides a quantitative version of Serre vanishing theorem
for the associated sheaf.
It in particular bounds the largest degree of the minimal generators and the smallest twist for which the sheaf is generated by its global sections. It has been used as a measure for the complexity of computational problems in algebraic geometry and commutative algebra (see for example \cite{EG84} or \cite{BaMum}).

The two most frequent definitions of $\ZZ$-graded Castelnuovo-Mumford regularity are the one in terms of graded Betti numbers and the one using local cohomology. The equivalence of this two definitions is one of the main basic results of the theory. For a wider discussion about regularity, we refer to  \cite{MumRed} or to the survey of Bayer and Mumford \cite{BaMum} and the more recent one \cite{Ch07}.

\medskip

A multigraded extension of Castelnuovo-Mumford regularity for modules over a polynomial ring over a field was introduced by Hoffman and Wang in a special case in \cite{HW04}, and later by Maclagan and Smith in a more general setting in \cite{MlS04}. 

The main motivation for studying regularity over multigraded polynomial rings was from toric geometry. For a toric variety $X$ associated to a fan $\Delta$, the homogeneous coordinate ring, introduced in \cite{Cox95}, is a polynomial ring $R$ graded by the divisor class group $G$ of $X$ together with a monomial ideal $B_\Delta$ generated by monomials corresponding to the complement of faces in $\Delta$. The dictionary linking the geometry of $X$ with the theory of $G$-graded $R$\nobreakdash-modules leads to geometric interpretations and applications for multigraded regularity. 

In \cite{HW04}, Hoffman and Wang define the concept of regularity for bigraded modules over a bigraded polynomial ring motivated by the geometry of $\P1\times \P1$. They prove analogs of some of the classical results on $m$-regularity for graded modules over polynomial algebras. 
In \cite{MlS04}, Maclagan and Smith develop a multigraded variant of Castelnuovo-Mumford regularity also motivated by toric geometry, working with $G$-graded modules over a Cox ring $(S,G,B_\Delta )$.
\medskip

In this article we introduce a further generalization of this notion, first by working over any commutative base ring, and 
second by considering local cohomology with support in any finitely generated graded ideal. The first extension provides a
natural approach for working with families of sheaves or of graded modules and the second provides a more flexible and powerful
tool for applications.

Our definition of multigraded regularity is given in terms of the vanishing of graded components of local cohomology with respect to a graded ideal $B$, following \cite{HW04} and \cite{MlS04}.

We show how to transfer knowledge 
in two directions. First, to deduce some information on the graded Betti numbers from the knowledge of regions
where the local cohomology with support in a given graded ideal vanishes.
Second to assert the vanishing of local cohomology of a graded module with support in any graded ideal from the shifts in a free resolution of the graded module and the cohomology of the polynomial ring. 

The second direction was shown in \cite[7.2]{MlS04} --where they remark, besides their proof, the result directly follows from a standard spectral sequence argument--
but there was only partial results concerning the reverse direction. 

To go from regularity to resolutions, Maclagan and Smith proved the existence of a complex 
with $B$-torsion homology and shifts bounded in terms of $\reg_B (M)$, that provides a resolution at the sheaf level of
a truncation of $M$ giving rise to the same sheaf as $M$.
Our results give estimates for the shifts in a true resolution of the module itself. They are sharper and valid for a much larger class of ideals $B$.

In the standard $\ZZ^n$-graded case, H\`a \cite{Ha07} and Sidman {\it et al.} \cite{STW06} gave estimates for the multigraded Betti numbers in terms of two different notions of regularity.
Their results follow from our approach (with $G=\ZZ$): in one case taking as base ring $S=k[\X_2,\hdots,\X_{n}]$, $R=S[\X_1]$, graded by the degree in $\X_1$ and $B=R_+$, to estimate $\deg_{\X_1}$ of the syzygies (and similarly for the other groups of variables). In the other case by taking $S=k$ and coarsening the degree using a linear form $l:\ZZ^n\to G$, with $B$ generated by all variables.

\medskip
This article is organized in three sections.
In section $2$ we gather some basic facts on local cohomology. Local
cohomology, defined as the cohomology of the \v Cech complex constructed on a finite set of generators of $B$ only depend on the radical of the ideal $B$, and correspond the sheaf cohomology with support in $V(B)$ (see \cite[Chap.\ $2.3$]{HartLC67} or \cite{CJR}) but not always coincide with the right derived functors of the left exact functor $H^0_B$. We show that both notions coincide when $B$ is a finitely generated monomial ideal in a polynomial ring, hence $H^i_B(M)$ coincides 
with $\colimit{t}\Ext^i_R(R/B^t,M)$, and the approach of Musta{\c{t}}\u{a} in \cite{Mus00} provides a way of computing 
this limit for a monomial ideal.

The key result in Section $3$ is Theorem \ref{ThmRegHGral} that allows to deduce vanishing of Tor modules from
vanishing of local cohomology, and vice-versa.
More precisely, Theorem \ref{ThmRegHGral} applied to  a graded free resolution of a module $M$, provides bounds for the supports of $H^i_B(M)$ in terms of the shifts in the resolution and vanishing regions for the local cohomology of the polynomial ring (see for example Theorem \ref{lemSuppHi}).
The same result applied to a Koszul complex, gives bounds for the supports of Tor modules in terms of the vanishing regions of the modules $H^i_B(M)$ (see for example Theorem \ref{ThmLCtoTor} and \ref{CorLCtoTor}). As in the standard $\ZZ$-graded case, a complex, extending a presentation, that has positive homology of small cohomological dimension can be used in place of a resolution to bound the support of local cohomology (and hence the regularity). The key results on persistence of local cohomology 
vanishing are derived by very similar arguments (see Lemma \ref{wR} and Corollary \ref{pers1dir}).

In Section $4$ we give the definition of regularity and we refine serval results obtained by D. Maclagan, H.\ T.\ H\`a, J.\ Sidman, G.\ Smith, B.\ Strunk, A.\ Van Tuyl and H.\ H.\ Wang, in \cite{MlS04, STW06, ST06, Ha07, HS07}. Before illustrating some differences on an example of Sidman {\it et al.} \cite[Ex.\ 1.1]{STW06}, we shall give
some definitions and notations.

Let $S$ be a commutative ring, $\G$ an abelian group and $R:=S[X_1,\hdots,X_n]$, with $\deg(X_i)=\gamma_i$ and $\deg (s)=0$ for $s\in S$. Consider $B\subseteq (X_1,\hdots,X_n)$ a finitely generated graded $R$-ideal and $\Cc$  the monoid generated by $\{\gamma_1,\hdots, \gamma_n\}$, we propose in Definition \ref{defRegLC} that for $\gamma\in \G$, and for a graded $R$-module $M$ is  \textsl{weakly $\gamma$-regular} if 
\[
 \gamma \not\in \bigcup_{i} \Supp_\G(H^i_B(M))+\FF_{i-1}.
\] 
where   $\FF_i :=\{ \gamma_{j_1}+\cdots + \gamma_{j_i}\ \vert\ j_1\leq\cdots \leq j_i\}$ for $i>0$,  $\FF_{0}:=\{ 0\}$, $\FF_{-1}:=-\FF_1$ and $\FF_i=\emptyset$ else.
If further, $M$ is weakly $\gamma'$-regular for any $\gamma'\in \gamma +\Cc$, then $M$ is
\textsl{$\gamma$-regular} and 
\[
 \reg_B (M):=\{ \gamma\in \G\ \vert\ M\ {\rm is}\ \gamma {\rm -regular}\} .
\]

It follows from the definition that $\reg_B (M)$ is the maximal set $\SS$ of elements in $\G$ such that $\SS +\Cc =\SS$ and $M$ is $\gamma$-regular for any $\gamma\in \SS$.


\begin{exmp}[Example 1.1 in \cite{STW06}]\label{ex11}
Take $I=(X_0X_1,Y_0Y_1)$ a complete intersection ideal in $R=k[X_0,X_1,Y_0,Y_1]$. If $\mm_X=(X_0,X_1)$, $\mm_Y=(Y_0,Y_1)$, $R_+=\mm_X + \mm_Y$, and $B=\mm_X \cap \mm_Y$, they show that $\reg_B(R/I)=(1,1)+\ZZ^2_{\geq 0}$. This can be obtained, as the authors observe, by computing the Hilbert function of $R/I$ and applying Proposition 6.7 in \cite{MlS04}, or by relating $H^i_B(R/I)_\mu$ with sheaf cohomology  (cf.\ \cite{Mus02}), or just 
from Mayer-Vietoris exact sequence.

Since $I$ is a complete intersection, $R/I$ is Cohen-Macaulay of dimension two, hence $H^i_{R_+}(R/I)=0$ for all $i\neq 2$ and $H^2_{R_+}(R/I)=\omega_{R/I}^*$.

\begin{minipage}{8cm}
We get that $\reg_{R^+}(R/I)$ is the $\ZZ_{\geq 0}^2$-stable part of 
\[
 \complement (\Supp_G(\omega_{R/I}^*)+\FF_1),
\]

where $\Supp_G(\omega_{R/I}^*) \subseteq \ZZ_{\leq 0}^2$. 

Thus,
\[
 \reg_{R^+}(R/I)=(\ZZ_{\geq 2}\times \ZZ) \cup(\ZZ\times \ZZ_{\geq 2}) \cup\{(1,1)\},
\]
as is shown on the right.

\end{minipage}
\   \
\hfill
\begin{minipage}{5.8cm}
\includegraphics{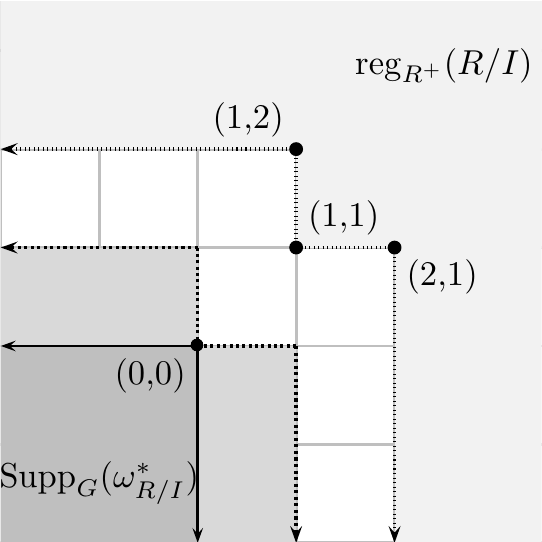}

\end{minipage}

Theorem \ref{ThmLCtoTor2} gives that 
\[
 \Supp_{\ZZ^2}(\tor^R_j(R/I,k)) \subset ((1,0) +\EE_j+\complement \reg_{R^+} (R/I))\cap ((0,1) +\EE_j+\complement \reg_{R^+} (R/I))
\]

which is described in the figure on the left, while applying Theorem \ref{ThmLCtoTor2} to $\reg_{B} (R/I)$ one would get a less sharp region described in the figure on the right:

\begin{center}
 \includegraphics{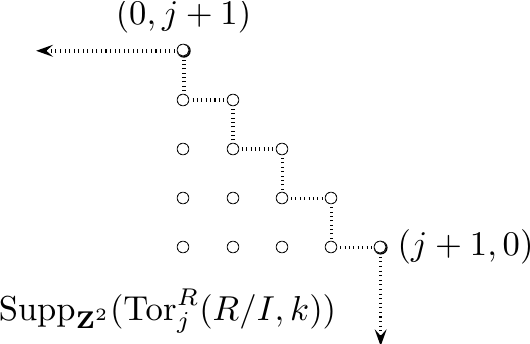}\qquad  \includegraphics{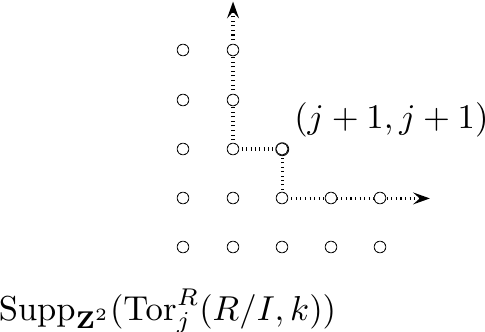}
\end{center}

We see in this example that the estimate of Betti numbers is much sharper 
using $R_+$-regularity in place of $B$-regularity. This is a general phenomenon.
However, it is not always easy to compute or estimate the $R_+$-regularity (which is 
defined for any grading, as well as the one relative to any monomial ideal). In a
toric setting, the coarser information on Betti numbers given by the regularity relative 
to the irrelevant ideal $B$ attached to the Cox ring is often more accessible (see \cite{Bot10} 
for an example in this direction). 

The fact that one can, in many cases, provide finite regions for the support of Tor modules from $R_+$-regularity is explained in Remark \ref{remSuppFiniteR+}. This in particular applies to projective toric varieties. 
\end{exmp}

The results of section 3 on persistence of vanishing properties for local cohomology gives Theorem \ref{wRtoR} that extend the fact that, in the classical case, weakly regular implies regular. Our argument provides a very short proof of this fact in the classical case, that was first established by
Mumford in \cite{MumRed} for finitely generated modules, with a field as base ring, and in full generality by Jouanolou (see \cite[2.2]{CJR}).
Our persistence result extend and refine the results of \cite{MlS04} on this issue. They show that the vanishing of local cohomology modules in a finite number of homological and internal degrees provides a regularity criterion, as in the classical case.\\

The second and third subsections are dedicated to a detailed study of how one can pass information from
regularity to shifts in a free resolution and conversely. These results rely heavily on the material in the 
third section and are the core of this section. \\

Subsection $4.4$ is dedicated to extend results in \cite{MlS04} and \cite{STW06} providing sharper finite subsets of the grading group $G$ that bound the degrees of the minimal generators of a minimal free resolution of a truncation of $M$. In Lemma \ref{tronctor} we provide a multigraded variant for the bounds on the shifts in a minimal free resolution of $M_{\geq d}$ in the classical case. Here ``$M_{\geq d}$" is replaced by $M_{\Ss }$, $\Ss$ being a $\Cc$-stable subset of $G$.

The result  takes a simple form when $\Ss\subseteq \reg_B(M)$, giving Theorem \ref{ThmLCtoTor3}. In particular, taking $\Ss=\mu+\Cc$ with $\mu\in  \reg_B(M)$, we get as corollary Theorem 5.4 in \cite{MlS04}, as well as several results in \cite{STW06} and in Section 7 of \cite{MlS04}. Also notice that Lemma \ref{regtrunc} allows one to study regularity with respect to more general monoids 
than $\Cc$ in $G$, similarly as in the approach of Maclagan and Smith. 

The fifth subsection of Section $4$ consists in a very simple example where we illustrate that the vanishing of the local cohomology modules of a principal ideal depends not only on the degree but on the generator itself, opposite to what happens in the classical theory. For simplicity, we treat the case of a form of bidegree $(1,1)$, with $R$ and $B$ as in Example \ref{ex11}. The vanishing of local cohomology depends on the factorization of the form, and the same kind of phenomenon occur for any bidegree. 
In this example  the bounds obtained on the support of Tor modules are pretty sharp. It shows that a  $R$-module with regularity $\ZZ^2_{\geq 0}=\reg_B(R)$ can have a first syzygy of bidegree $(1,1)$. It also shows that two modules with the same resolution may have different regularity at level $2$, and we leave the following open question:

\begin{ques}
Is there a ring $R$, an ideal $B$ and two modules $M$ and $N$ satisfying: $\Supp_\G(\tor^R_i(M,k))=\Supp_\G(\tor^R_i(N,k))$ for all $i$, and $\reg_B(M)\neq \reg_B(N)$?
\end{ques}

Castelnuovo-Mumford regularity has applications to Hilbert functions of graded modules, to which we dedicate the last subsection of this paper. Such questions comes intrinsically from the algebraic perspective, but also motivated by the geometry behind. 
  
The study of Hilbert functions over standard graded algebras has taken a central role in commutative algebra and algebraic geometry since the famous paper of Hilbert \cite{Hil90} in 1890.

The Lemma \ref{lemClassCC} is the key ingredient of a short proof of Grothendieck-Serre formula for standard multigraded polynomial rings, that was proved in \cite[Prop.\ 2.4.3]{ColTesis}). This lemma shows in particular that if the function $F_R$ (see \ref{secHF} \eqref{eqFM}) of a multigraded polynomial ring $R$ belongs to a given class, closed under shifts and addition, then, so does $F_M$ for any finitely generated $R$-module $M$. The computation of $F_R$ for a given grading and monomial ideal $B$ 
is a simple task --the standard multigraded case follows-- so that this method also allows to treat any given $G$ and monomial ideal $B$.
\medskip

\section{Local Cohomology}

Let $R$ be a commutative ring and $B$ be a finitely generated ideal. One can define the local
cohomology groups of an $R$-module $M$ as the homologies of the \v Cech complex constructed on a finite set of generators of $B$.
These homology groups only depend on the radical of the ideal $B$, and correspond the sheaf cohomology with support in 
$V(B)$ (see \cite[Chap.\ $2.3$]{HartLC67} or \cite{CJR}). This in particular implies that one has a Mayer-Vietoris sequence for \v Cech
cohomology. This cohomology commutes with arbitrary direct sums. It coincide with the right derived functors of
the left exact functor $H^0_B$ in several instances (notably when $R$ is Noetherian or $B$ is generated by a regular sequence in $R$). From the Mayer-Vietoris sequence, it also follows that both coincides when $B$ is a 
finitely generated monomial ideal in a polynomial ring (see below).

\medskip

\subsection{Local cohomology with support on monomial ideals}

In this section we study the support of local cohomology modules with support on a monomial ideal $B$. Let $R:=S[X_1,\hdots,X_n]$ be a polynomial ring over a commutative ring $S$, $\deg(X_i)=\gamma_i\in \G$ for $1\leq i\leq n$ and $\deg (s)=0$ for $s\in S$. 

Assume that $B\subseteq (X_1,\hdots,X_n)$ is a monomial $R$-ideal. Then $B$ is finitely generated, and since local cohomology with support in the monomial ideals ideal $B$ and $\sqrt B$ coincide, we can assume that $B=\sqrt B$, hence $B=\bigcap_{i=1}^t J_i$, where $J_i=(X_{i_1},\hdots,X_{i_{s(i)}})$ is an $R$-ideal. A motivating example is the Cox ring of a toric variety,  see \cite{Cox95}.

\begin{lem}\label{lemLCMonomial}Let $M$ be a graded $R$-module, then
\begin{equation}\label{eqLCMonomial}
\Supp_\G(H_B^\ell (M)) \subset \bigcup_{1\leq i \leq t}\bigcup_{1\leq j_1<\cdots<j_i \leq t} \Supp_\G(H^{\ell+i-1}_{J_{j_1}+\cdots+J_{j_i}}(M)).
\end{equation}
 
\end{lem}
\begin{proof}
 Let $B=\bigcap_{i=1}^t J_i$. We induct on $t$. The result is obvious for $t=1$, thus, assume that $t> 1$ and that \eqref{eqLCMonomial} holds for $t-1$. Write $J_{\leq t-1}:= J_1\cap\cdots\cap J_{t-1}$. The, for $t>1$ and $\ell\geq 0$ consider the Mayer-Vietoris long exact sequence of local cohomology 
\begin{equation*}\label{eqMVMonomial}
 \cdots \to H^\ell_{J_{\leq t-1}+J_t}(M) \to  H^\ell_{J_{\leq t-1}}(M)\oplus H^\ell_{J_t}(M) \to  H^\ell_{B}(M) \to  H^{\ell+1}_{J_{\leq t-1}+J_t}(M) \to \cdots.
\end{equation*}
Hence, $\Supp_\G(H_B^\ell (M)) \subset \Supp_\G(H^\ell_{J_{\leq t-1}}(M))\cup\Supp_\G(H^\ell_{J_t}(M)) \cup\Supp_\G(H^{\ell+1}_{J_{\leq t-1}+J_t} (M))$. By inductive hypothesis 
\[
\Supp_\G(H^\ell_{J_{\leq t-1}}(M)) \subset \bigcup_{1\leq i \leq t-1}\bigcup_{1\leq j_1<\cdots<j_i \leq t-1} \Supp_\G(H^{\ell+i-1}_{J_{j_1}+\cdots+J_{j_i}}(M)).
\] Since $J_{\leq t-1}+J_t=(J_1+J_t)\cap\cdots\cap (J_{t-1}+J_t)$, again by inductive hypothesis we obtain that $\Supp_\G(H^{\ell+1}_{J_{\leq t-1}+J_t}(M))\subset \bigcup_{1\leq i \leq t-1}\bigcup_{1\leq j_1<\cdots<j_i \leq t-1} \Supp_\G(H^{\ell+i-1}_{J_{j_1}+\cdots+J_{j_i}+J_{j_t}}(M))$ which complets the proof.
\end{proof}

\begin{rem}
The exact sequence 
\begin{equation*}
  H^\ell_{J_1\cap\cdots\cap J_{t-1}}(M)\oplus H^\ell_{J_t}(M) \to  H^\ell_{B}(M) \to  H^{\ell+1}_{(J_1+J_t)\cap\cdots\cap (J_{t-1}+J_t)}(M) 
\end{equation*}
applied for $\ell\geq 1$ and $M$ injective shows, by recursion on $t$, that $H^\ell_{B}(M)=0$ in this 
case (the case $t=1$ is classical and follows from the fact that $B$ is then generated by a regular sequence).
This in turn shows the following result (by reduction reduction to the case of polynomial rings
in finitely many variables) 
\end{rem}

\begin{thm}
Let $R$ be a polynomial ring over a commutative ring $S$ and $B$ be a finitely generated monomial $R$-ideal.
Then the \v Cech cohomology functor $H^\ell_{B}$ is the $\ell$-th right derived functor of $H^0_{B}$.
\end{thm}

The approach of Musta{\c{t}}\u{a} in \cite{Mus00} that we now recall uses the isomorphism 
$$
H^i_B(M)\simeq \colimit{t}\Ext^i_R(R/B^t,M)
$$
which holds over any commutative ring, taking for $H^i_B$ the $i$-th derived functor of $H^0_B$. As for 
a monomial ideal $B$ this agrees with \v Cech cohomology we have an isomorphism in our setting. Let
$B^{[t]}:=(f_1^t,\ldots ,f_s^t)$ where the $f_i$'s are the minimal monomial generators of $B$, the Taylor
resolution $T^{t}_{\bullet}$ of $R/B^{[t]}$ has a natural map to the one of $R/B^{[t']}$ for $t\geq t'$ that in turn 
provides a natural map: $\Hom_R(T^{t'}_{\bullet},R)\rightarrow \Hom_R(T^{t}_{\bullet},R)$. This ${\ZZ}^n$-graded map is
an isomorphism of complexes in degree $\gamma \in (-t',\ldots ,-t')+{\ZZ}_{\geq 0}^n$ and else 
$\Hom_R(T^{t'}_{\bullet},R)_\gamma =0$. 
For $a=(a_1,\cdots ,a_n)\in \{ 0,1\}^n$, let $E_{a}:=\{ i,\ a_i =0\}$ and $R^*_a ={\frac{1}{X^a}}S[X_i,X_j^{-1},i\in E_a, j\in \{ 1,\ldots ,n\}\setminus E_a ]$.

Setting $N_{i,a}:=H^i (\Hom_R(T_\bullet ,R)_{-a})$, were $T_\bullet :=T_\bullet^1$ is the Taylor resolution of $R/B$, by Musta{\c{t}}\u{a} description one has:
$$
H^i_B (R)=\oplus_{a\in \{ 0,1\}^n}H^i_B (R)_{-a}\otimes_S R^*_a=\oplus_{a\in \{ 0,1\}^n}N_{i,a}\otimes_S R^*_a .
$$
Furthermore, this sum is restricted by the inclusion \eqref{eqLCMonomial} or by inspecting a little $T^1_\bullet$. For instance if $n-\# E_a =\vert a \vert < i$ then $N_{i,a}=0$.


\medskip

\section{Local Cohomology and graded Betti numbers}
In this chapter we aim is to establish a clear relation between supports of local cohomology modules and supports of $\tor$ modules and Betti numbers, in order to give a general definition for Castelnuovo-Mumford regularity in next chapter.

Throughout this chapter, $\G$ is a finitely generated abelian group, $R$ is a commutative $\G$-graded ring with unit and $B$ is a finitely generated homogeneous $R$-ideal. 

\begin{rem}\label{remGrading}
 Is of particular interest the case where $R$ is a polynomial ring in $n$ variables over a commutative ring whose elements have degree 0 and $\G=\ZZ^n/K$, is a quotient of $\ZZ^n$ by some subgroup $K$. Note that, if $M$ is a $\ZZ^n$-graded module over a $\ZZ^n$-graded ring, and $\G=\ZZ^n/K$, we can give to $M$ a $\G$-grading coarser than its $\ZZ^n$-grading. For this, define the $\G$-grading on $M$ by setting, for each $\gamma\in \G$, $M_\gamma:=\bigoplus_{d\in \pi^{-1}(\gamma)}M_d$.
\end{rem}

\medskip
In order to fix the notation, we state the following definitions concerning local cohomology of graded modules, and support of a graded modules $M$ on $\G$. Recall that the cohomological dimension of a module $M$ is $\cd_B(M)
:=\inf \{ i\ \vert\ H^j_B(M)= 0, \forall j>i\}$.  

\begin{defn}\label{defSuppGP}
Let $M$ be a graded $R$-module, the support of the module $M$ is $\Supp_\G(M):=\{\gamma \in \G:\ M_\gamma \neq 0\}$.
\end{defn}

If $\F.$ is a free resolution of a graded module $M$, much information on the module can be read from the one of the resolution. 
It has been observed by Gruson, Lazarsfed and Peskine that  a complex which need not be a resolution of $M$, but $M$ is its first non-vanishing homology, can be used in place of a resolution in some circumstances. Our next result is following this line
of ideas. We first give a definition.

\begin{defn}\label{defDij}
Let $\C.$ be a complex of graded $R$-modules. For all $i,j\in \ZZ$ we define a condition \eqref{eqDij} as follows
\begin{equation}\label{eqDij}\tag{D$_{ij}$}
 H^{i}_B(H_{j}(\C.))\neq 0 \textnormal{ implies } H^{i+\ell +1}_B(H_{j+\ell}(\C.))=H^{i-\ell -1}_B(H_{j-\ell }(\C.))=0 \textnormal{ for all }\ell\geq 1.
\end{equation}
\end{defn}

The following result provides information on the support of the local cohomology modules of the homologies of $\C.$ assuming \eqref{eqDij}.

\begin{thm}\label{ThmRegHGral}
 Let $\C.$ be a complex of graded $R$-modules and $i\in \ZZ$. If \eqref{eqDij} holds, then
\[
 \Supp_\G(H^i_B(H_j(\C.)))\subset \bigcup_{k\in\ZZ}\Supp_\G(H^{i+k}_B(C_{j+k})).
\]
\end{thm}
\begin{proof}
Consider the two spectral sequences that arise from the double complex $\check \Cc^\bullet_B \C.$ of graded $R$-modules. 

The first spectral sequence has as second screen $ _2'E^i_j = H^i_B (H_j(\C.))$. Condition \eqref{eqDij} implies that $_\infty 'E^i_j =\, _2'E^i_j = H^i_B (H_j(\C.))$. The second spectral sequence has as first screen $ _1''E^i_j = H^i_B (C_j)$. 

By comparing both spectral sequences, one deduces that, for all $\gamma \in \G$, the vanishing of $(H^{i+k}_B (C_{j+k}))_\gamma$ for all $k$ implies the vanishing of $( _\infty 'E^{i+\ell }_{j+\ell})_\gamma$ for all $\ell$, which carries the vanishing of $(H^i_B (H_{j}(\C.)))_\gamma$.
\end{proof}

We next give some cohomological conditions on the complex $\C.$ to imply \eqref{eqDij} of Definition \ref{defDij}. 

\begin{lem}\label{RemSuppCond}
Let $\C.$ be a complex of graded $R$-modules. Consider the following conditions
\begin{enumerate}
 \item $\C.$ is a right-bounded complex, say $C_j=0$ for $j<0$ and, $\cd_B(H_j(\C.))\leq 1$ for all $j\neq 0$.
 \item For some $q\in \ZZ\cup \{-\infty\}$, $H_j(\C.)=0$ for all $j<q$ and, $\cd_B(H_j(\C.))\leq 1$ for all $j>q$.  
 \item $H_j(\C.)=0$ for $j<0$ and $\cd_B(H_k(\C.))\leq k+i$ for all $k\geq 1$.
\end{enumerate}
Then, 
\begin{enumerate}
 \item[(i)] $(1)\Rightarrow (2)\Rightarrow$ \eqref{eqDij} for all $i,j\in \ZZ$, and
 \item[(ii)] $(1)\Rightarrow (3)\Rightarrow$\eqref{eqDij} for  $j=0$.
\end{enumerate}
\end{lem}
\begin{proof}
 For proving item (i), it suffices to show that  $(2)\Rightarrow$ \eqref{eqDij} for all $i,j\in \ZZ$ since $(1)\Rightarrow (2)$ is clear. 
 
 Let $\ell \geq 1$.
 
 Condition (2) implies that $H^i_B (H_j(\C.))=0$ for $j>q$ and $i\not= 0,1$ and for $j<q$.  If $H^i_B (H_j(\C.))\not= 0$, either $j>q$ and $i\in \{ 0,1\}$ in
 which case $j+\ell>q$ and $i+\ell +1\geq 2$ and $i-\ell -1<0$, or $j=q$ in
 which case $j+\ell>q$ and $i+\ell +1\geq 2$ and $j-\ell<0$. In both cases the asserted vanishing holds.

Condition $(1)$ automatically implies $(3)$. Condition $(3)$ implies that $H^{i+\ell +1}_B(H_{\ell}(\C.))=0$ and  $H_{j-\ell}(\C. )=0$.
\end{proof}

In the following subsection we establish the relation between the support of local cohomology modules and support of Tor modules, 
applying Theorem \ref{ThmRegHGral} and Lemma \ref{RemSuppCond} to a Koszul complex.

\subsection{From Local Cohomology to Betti numbers}\label{LCtoBN}
In this subsection we bound the support of Tor modules in terms of the support of local cohomology modules. This generalizes the fact that for $\ZZ$-graded Castelnuovo-Mumford regularity, setting $b_i(M) := \max\{\mu\ |\ \tor^R_i(M, k)_\mu \neq 0\}$ and $a_i(M) := \max\{\mu\ |\ H^i_\mm (M)_\mu \neq 0\}$, one has $b_i(M)-i \leq \reg(M) := \max_i\{a_i(M) + i\}$.

\medskip

Assume $R:=S[X_1,\hdots,X_n]$ is a polynomial ring over a commutative ring $S$, 
$\deg(X_i)=\gamma_i\in \G$ for $1\leq i\leq n$ and $\deg (s)=0$ for $s\in S$. 

Let $B\subseteq (X_1,\hdots,X_n)$ be a  finitely generated graded $R$-ideal.

\begin{nota}\label{RemShiftPSigma}
 For an $R$-module $M$, we denote by $M[\gamma']$ the shifted module by $\gamma'\in \G$, with $M[\gamma']_{\gamma}:=M_{\gamma'+\gamma}$ for all $\gamma \in \G$. 
\end{nota}

Let $M$ be a graded $R$-module, $f:=(f_1,\hdots,f_r)$ be a $r$-tuple of homogeneous elements of $R$ and $I$ the $R$-ideal generated by the $f_i$'s. Write $\k.(f;M)$ for the Koszul complex of the sequence $(f_1,\hdots,f_r)$ with coefficients in $M$. 
The Koszul complex $\k.(f;M)$ is graded as well as its homology modules $\h.(f;M)$. Set $\d_i:=\deg(f_i)$, $\EEE_0^f :=\{ 0\}$ and $\EEE_i^f :=\{ \d_{j_1}+\cdots +\d_{j_i},  j_1<\cdots <j_i\}$.

\begin{thm}\label{ThmLCtoTor}
If  $B\subset \sqrt{I+\ann_R(M)}$, then 
\[
\Supp_\G(\HHH_j(f;M))\subset \bigcup_{k\geq 0}(\Supp_\G(H^{k}_B(M))+\EEE_{j+k}^f),
\]
for all $j\geq 0$.
\end{thm}
\begin{proof}
We first notice that $\HHH_j(f;M)$  is annihilated by $I+\ann_R (M)$, hence it has cohomological dimension $0$ relatively to $B$. 
 
 According to Lemma \ref{RemSuppCond} (case (1)), Theorem \ref{ThmRegHGral} applies and shows that 
 \[
 \Supp_\G(\HHH_j(f;M))\subset \bigcup_{\ell\geq 0}\Supp_\G(H^{\ell}_B(K_{j+\ell }(f;M)))=\bigcup_{k\geq 0}(\Supp_\G(H^{k}_B(M))+\EEE_{j+k}^f). \qedhere
\]
 \end{proof}

 Further notice that, by \cite[Lem.\ 4.6]{CJR},  $\cd_B (N)\leq \cd_B (M)$ if $M$ is finitely presented and $\Supp_R (N)\subseteq \Supp_R (M)$;
which implies that Theorem \ref{ThmLCtoTor} holds if $\cd_B (M/IM)=0$ and $M$ is finitely presented.

In particular, taking $r=n$ and $f_i=X_i$ for all $i$, this establishes a relationship between the support of the local cohomologies  and the graded Betti numbers of a module $M$.  
\begin{cor}\label{CorLCtoTor}
For any integer $j$, set $X:=(X_1,\ldots ,X_n)$, then
\[
\Supp_\G(\tor^R_j(M,S))\subset \bigcup_{k\geq 0}(\Supp_\G(H^{k}_B(M))+\EEE^X_{j+k}).
\]

\end{cor}

Notice that taking $G=\ZZ$ and $\deg(X_i)=1$, Corollary \ref{CorLCtoTor} gives the well know bound $b_i(M)-i \leq \reg(M) := \max_i\{a_i(M) + i\}$.

The following lemma about persistence of cohomology vanishing contains the fact that for the standard ${\bf Z}$-grading of $S$ the
notions of weak and strong regularity coincides.

\begin{lem}\label{wR}
Let  
$\ell$ be an integer. If $\ell >\cd_B (R/(I+\ann_R (M)))$,
\[
 \gamma \not\in \bigcup_{i\geq 0} \Supp_\G(H^{\ell +i}_B(M))+\EEE_{i+1}^I\ \Rightarrow  \gamma \not\in 
 \Supp_\G(H^{\ell }_B(M)).
\] 

 If $\ell =\cd_B (R/(I+\ann_R (M)))$ and $\gamma \not\in \bigcup_{i>0} \Supp_\G(H^{\ell +i}_B(M))+\EEE_{i+1}^I$, then
 \[
(H^{\ell }_B(M)/IH^{\ell }_B(M))_\gamma \subseteq H^{\ell }_B(M/IM)_\gamma 
\] 
and equality holds if  $\gamma \not\in \bigcup_{i> 0} \Supp_\G(H^{\ell +i}_B(M))+\EEE_{i}^I$.
\end{lem}

\begin{proof}  Consider the two spectral sequences that arise from the double \v Cech-Koszul complex $\check \Cc^\bullet_B \k. (f;M)$ of graded $R$-modules. 

The first spectral sequence has as second screen 
$ _2'E^i_j = H^i_B(\HHH_j (f;M))$. 
As $I$ and $\ann_R (M)$ annihilate $\HHH_j (f;M)$, 
$\cd_B (\HHH_j (f;M))\leq \cd_B (R/(I+\ann_R (M))) <\ell$, which shows that $_2'E^i_j =0$ for $i-j=\ell$ unless 
$\ell =\cd_B (M/IM)=\cd_B (R/(I+\ann_R (M)))$, in which case $_2'E^i_j =0$ for $j\not= 0$ and $_2'E^\ell_0 =_\infty 'E^\ell_0 =H^{\ell }_B(M/IM)$.   The second spectral sequence has as first screen $ (_1''E^i_j)_\mu = \oplus_{\gamma\in\EEE_j^I}H^i_B (M)_{\mu-\gamma}^{b_{j,\gamma}}$ for some positive $b_{j,\gamma}\in {\ZZ}$ ($b_{00}=1$). 

By hypothesis $(_1''E^{\ell +i }_{i+1})_\mu =0$ for all $i\geq 0$. As  $(_1''E^{\ell -i }_{-i-1})=0$ for $i\geq 0$, we deduce that $ (_1''E^\ell_0)_\mu =  (_\infty''E^\ell_0)_\mu$. As $ (_1''E^\ell_0)_\mu = H^\ell_B (M)_{\mu}$ and 
$_\infty ''E^i_j =_\infty 'E^i_j =0$ for $i-j=\ell$, the conclusion follows.
\end{proof}

The following special case gives a persistence criterion for local cohomology vanishing that will be used to give cases where weak regularity implies regularity.

\begin{cor}\label{pers1dir}
Let $\ell \geq 1$ be an integer $\gamma \in\G$ and assume that $B\subseteq  \sqrt{(R_\gamma )+\ann_R(M)}$.
Then, for any $\mu \in G$,

\begin{enumerate}
\item $H^{\ell +i}_B(M)_{\mu +i\gamma }=0,\forall i\geq 0\quad \Rightarrow\quad H^{\ell +i}_B(M)_{(\mu +\gamma )+i\gamma }=0,\forall i\geq 0$.

\item If  $H^{i+1}_B(M)_{\mu +i\gamma }=0,\forall i\geq 0$, then
\[
M_{\mu +j\gamma }/R_{\gamma}M_{\mu +(j-1)\gamma }=H^0_B(M)_{\mu +j\gamma }/R_{\gamma}H^0_B(M)_{\mu +(j-1)\gamma }, \forall j>0.
\]
\end{enumerate}
\end{cor}

\subsection{From Betti numbers to Local Cohomology}

In this subsection we bound the support of local cohomology modules in terms of the support of Tor modules. This generalizes the fact that for $\ZZ$-graded Castelnuovo-Mumford regularity, if $a_i(M)+i \leq \reg(M) := \max_i\{b_i(M) - i\}$. 

We keep same hypotheses and notation as in Section \ref{LCtoBN}

\medskip
Next result gives an estimate of the support of local cohomology modules of a graded $R$-module $M$ in terms of the supports of those of base ring and the twists in a free resolution. This, combined with Lemma \ref{LemResolutions}, gives an estimate for the support of local cohomology modules in terms of Betti numbers. 

The key technical point is that Lemma \ref{LemResolutions} part (1) and (2) give a general version of Nakayama Lemma in order to relate shifts in a resolution with support of Tor modules; while part (3) is devoted to give a `base change lemma' in order to pass easily to localization.

\begin{thm}\label{lemSuppHi}
 Let $M$ be a graded $R$-module and $F_\bullet$ be a graded complex of free $R$-modules, with $F_i=0$ for $i<0$ and $H_0(F_\bullet)=M$. Write 
 $F_i = \bigoplus_{j\in E_i} R[-\gamma_{ij}]$ and $T_i:=\{ \gamma_{ij}\ \vert\ j\in E_i\}$. Let $\ell \geq 0$ and assume $\cd_B(H_j(F_\bullet ))\leq \ell +j$ for all $j\geq 1$. Then, 
\[
 \Supp_{\G}(H^\ell_B(M))\subset \bigcup_{i\geq 0}(\Supp_{\G}(H^{\ell +i}_B(R))+T_i).
\]
\end{thm}
\begin{proof} Lemma \ref{RemSuppCond} (case (3)) shows that  Theorem \ref{ThmRegHGral} applies for estimating the support of
local cohomologies of $H_0(F_\bullet)$, and provides the quoted result as local cohomology commutes with arbitrary direct sums 
\[
 \Supp_{\G} (H^{p}_B(R[-\gamma ]))=\Supp_{\G} (H^{p}_B(R))+\gamma, \mbox{ and }\Supp_{\G} (\oplus_{i\in E} N_i)=\cup_{i\in E}\Supp_{\G} (N_i)
\]
for any set of graded modules $N_i$, $i\in E$.
\end{proof}

\begin{lem}\label{LemResolutions}
 Let $M$ be a graded $R$-module.
 \begin{enumerate}
  \item Let $S$ be a field and let $F_\bullet$ be a $G$-graded free resolution of a finitely generated module $M$. Then 
   \[
   F_i=\bigoplus_{\gamma\in T_i} R[-\gamma]^{\beta_{i,\gamma}},\quad \mbox{and}\quad T_i= \Supp_G (\tor^R_i(M,S)).
   \] 
  \item Assume that there exists $\phi\in\Hom_{\ZZ}(G,\RR )$ such that $\phi (\deg (x_i))>0$ for all $i$. If $\phi (\deg (a) )>m$ for some $m\in \RR$ and any $a \in M$, then there exists a $G$-graded free resolution $F_\bullet$ of $M$ such that
\[
 F_i=\bigoplus_{j\in E_i}R[-\gamma_{ij}]\quad \hbox{with}\quad  
 \gamma_{ij} \in \bigcup_{0\leq \ell\leq i}\Supp_G (\tor^R_\ell(M,S))\ \forall j .
\]
If, furthermore, there exists $p$ such that $F_i$ is finitely generated for $i\leq p$, then  $E_i$ is finite for $i\leq p$.
  \item Assume that $(S,\mm ,k)$ is local. Then
 \[
 \Supp_{\G}(\tor_i^R(M,k))\subseteq \bigcup_{j\leq i}\Supp_{\G}(\tor_j^R(M,S)).
 \]
 \end{enumerate}
\end{lem}
\begin{proof}
Part (1) follows the standard arguments of the $\ZZ$-graded case. For part (2) see \cite[Prop.\ 2.4]{BCH} and its proof.
Part (3) follows from the fact that if $(S,\mm ,k)$ is local there is an spectral sequence $\tor_p^{S}(\tor_q^R(M,S),k)\Rightarrow \tor_{p+q}^R(M,k)$ and the fact that $S\subset R_0$. 
\end{proof}

Combining Theorem \ref{lemSuppHi} with Lemma \ref{LemResolutions} (case (1)) one obtains:

\begin{cor}\label{corSuppHi}
 Assume that $S$ is a field and let $M$ be a finitely generated graded $R$-module. Then, for any $\ell$,
\[
 \Supp_{\G}(H^\ell_B(M))\subset \bigcup_{i\geq 0}(\Supp_{\G}(H^{\ell +i}_B(R))+\Supp_{\G}(\tor_i^R(M,S))).
\]
\end{cor}

If $S$ is Noetherian, Lemma \ref{LemResolutions} (case (3)) implies the following:

\begin{cor}\label{corSuppHi2}
 Assume that $(S,\mm ,k)$ is local Noetherian and let $M$ be a finitely generated graded $R$-module. Then, for any $\ell$,
\[
\begin{array}{rl}
 \Supp_{\G}(H^\ell_B(M))&\subset \bigcup_{i\geq 0}(\Supp_{\G}(H^{\ell +i}_B(R))+\Supp_{\G}(\tor_i^R(M,k)))\\
 &\subset \bigcup_{i\geq j\geq 0}(\Supp_{\G}(H^{\ell +i}_B(R))+\Supp_{\G}(\tor_j^R(M,S))).\\
 \end{array}
\]
\end{cor}

After passing to localization, Corollary \ref{corSuppHi2} shows that:

\begin{cor}\label{corSuppHi3}
Let $M$ be a finitely generated graded $R$-module, with $S$ Noetherian. Then, for any $\ell$,
\[
 \Supp_{\G}(H^\ell_B(M))\subset \bigcup_{i\geq j\geq 0}(\Supp_{\G}(H^{\ell +i}_B(R))+\Supp_{\G}(\tor_j^R(M,S))).
 \]
\end{cor}
 
\begin{proof} 
Let $\gamma\in  \Supp_{\G}(H^\ell_B(M))$. Then $H^\ell_B(M)_\gamma \not= 0$, hence there exists 
 $\pp\in {\rm Spec} (S)$ such that $(H^\ell_B(M)_\gamma)\otimes_S S_\pp = H^\ell_{B\otimes_S S_\pp}(M\otimes_S S_\pp )\not= 0$.
 Applying Corollary \ref{corSuppHi2} the result follows since both the local cohomology functor and the
 Tor functor commute with localization in $S$, and preserves grading as $S\subset R_0$.
\end{proof}
  
 Finally, Lemma \ref{LemResolutions} (case (2))  gives:
 
\begin{cor}\label{corSuppHi4}
Let $M$ be a graded $R$-module, and assume that there exists $\phi\in\Hom_{\ZZ}(G,\RR )$ such that $\phi (\deg (x_i))>0$ for all $i$. If $\phi (\deg (a) )>m$ for some $m\in \RR$ and any $a \in M$, then, for any $\ell$,
\[
 \Supp_{\G}(H^\ell_B(M))\subset \bigcup_{i\geq j\geq 0}(\Supp_{\G}(H^{\ell +i}_B(R))+\Supp_{\G}(\tor_j^R(M,S))).
 \]
\end{cor} 

Notice that taking $G=\ZZ$ and $\deg(X_i)=1$, Corollaries \ref{corSuppHi}, \ref{corSuppHi2}, \ref{corSuppHi3} and \ref{corSuppHi4} give the well know bound $a_i(M)+i \leq \max_i\{b_i(M) - i\}$. 
 

\medskip

\section{Castelnuovo-Mumford regularity}\label{secCMR}

In this section we give a definition for a $\G$-graded $R$-module $M$ and $\gamma\in \G$ to be \textsl{weakly $\gamma$-regular} or just \textsl{$\gamma$-regular}, with respect  to a graded $R$-ideal $B$, depending if $\gamma$ is or is not on the shifted support of some local cohomology modules of $M$ with support in $B$ (cf.\ \ref{defRegLC}). 

The fact that weak regularity implies regularity in the classical case is generalized using Lemma \ref{wR}.  The corresponding results are given in Theorem \ref{wRtoR} that extends and refines  the results of \cite{MlS04} on this issue. It is proved that the vanishing of local cohomology modules in a finite number of homological and internal degrees provides a regularity criterion, as in the classical case.

The Castelnuovo-Mumford regularity of $\ZZ$-graded $R$-module $M$ is a cohomological invariant that bounds the degree of minimal generators of a minimal free resolution.  In the standard graded case, if $F_\bullet$ is a minimal graded free resolution of $M$, then the degrees of the generators of the modules $F_i$ are bounded above by $ \reg(M)+i$ (cf.\ Example \ref{exmpSuppTor}).

As we mentioned in the introduction, partial results were obtained in \cite{MlS04} and in \cite{STW06} and \cite{Ha07}. In  \cite{MlS04} by estimating, in the toric situation, the shifts in a resolution at the sheaf level, and in the other works by considering,
in special cases, 
variants of the definition of regularity. 
In Theorem \ref{ThmLCtoTor1} we provide an estimate for the support of $\tor^R_j(M,S)$. This estimate is refined 
in Theorem \ref{ThmLCtoTor2} 
under additional hypotheses that are often satisfied in the toric setting (results in  \cite{MlS04} are given in this 
situation).

Next, we provide bounds for the truncation $M$. Precisely, Lemma \ref{tronctor} gives a multigraded variant for the bound
on the shifts in a minimal free resolution of $M_{\geq d}$. Here ``$M_{\geq d}$" is replaced by $M_{\Ss }$, with $\Ss$ a $\Cc$-stable subset of $G$.

As in the graded case this results take particular interest when $d\geq \reg(M)$, which corresponds here to study 
$M_\Ss$ for $\Ss\subseteq \reg_B(M)$. This is done in Theorem \ref{ThmLCtoTor3}. In particular taking $\Ss=\mu+\Cc$, with $\mu\in  \reg_B(M)$ we get as corollary the results on truncation  of \cite{MlS04}, as well as several results in \cite{STW06}. 

An alternative way of getting this type of result, that works in many interesting cases, is through the 
computation of the local cohomology of the restriction of $M$ to some degrees. This is explained in Lemma
\ref{regtrunc} and Proposition \ref{propRegEll}.

We then study the example of a form of bidegree $(1,1)$ is a standard bigraded polynomial ring in four variables, over a field, corresponding to $\PP^1\times \PP^1$.  This extremely simple example already shows that the support of local cohomology 
depend upon the form, and illustrate some other features of $\ZZ^2$-graded regularity.

In a last part, we recall some results on Hilbert polynomials in the multigraded setting and connect them with our
results on regularity. We also prove a lemma that gives a short and elegant way to extend to standard multigrading
the classical theory of Hilbert polynomials. It can be used as well for the product of anisotropic projective spaces, 
and gives a way to handle any particular case.


\subsection{Definition of regularity and persistence of cohomological vanishing}

Let $S$ be a commutative ring, $\G$ an abelian group and $R:=S[X_1,\hdots,X_n]$, with $\deg(X_i)=\gamma_i\in G$ and $\deg (s)=0$ for $s\in S$. Denote by $\Cc$ the monoid generated by $\{\gamma_1,\hdots, \gamma_n\}$ and let
$B\subseteq R_+ :=(X_1,\hdots,X_n)$ be a finitely generated graded $R$-ideal.

\begin{defn}\label{defEE}
 Set $\EE_0:=\{0\}$, $\EE_l:=\{\gamma_{i_1}+\cdots+\gamma_{i_l}\ : \ i_1<\cdots<i_l\}$ for $l> 0$, $\EE_{-1}:=-\EE_1$
 and $\EE_l =\emptyset$ for $l<-1$.
\end{defn}

In addition to the definition of $\EE_i$, we introduce the following sets already used by Hoffman and Wang, Maclagan and Smith and other authors. For $i>0$,
$$
\FF_i :=\{ \gamma_{j_1}+\cdots + \gamma_{j_i}\ \vert\ j_1\leq\cdots \leq j_i\}
$$
and $\FF_{i}:=\EE_i$ for $i\leq 0$. It is clear that $\EE_i \subset \FF_i$.

Observe that if $\gamma_i=\gamma$ for all $i$, $\EE_l=\{l\cdot \gamma\}$ when $\EE_l \not= \emptyset$ and
$\FF_l=\{l\cdot \gamma\}$ when $\FF_l \not= \emptyset$.

\begin{defn}\label{defRegLC}
For $\gamma\in \G$ and $\ell\in \ZZ_{\geq 0}$, a graded $R$-module $M$ is  \textsl{very  weakly $\gamma$-regular at level $\ell$} if 
\[
 \gamma \not\in \bigcup_{i\geq \ell} \Supp_\G(H^i_B(M))+\EE_{i-1}.
\] 
$M$ is  very weakly \textsl{$\gamma$-regular} if it is  very weakly $\gamma$-regular at level 0.

$M$ is  \textsl{weakly $\gamma$-regular at level $\ell$} if 
\[
 \gamma \not\in \bigcup_{i\geq \ell} \Supp_\G(H^i_B(M))+\FF_{i-1}.
\] 
$M$ is  weakly \textsl{$\gamma$-regular} if it is  weakly $\gamma$-regular at level 0.

If further $M$ is weakly $\gamma'$-regular (resp. weakly $\gamma'$-regular at level $\ell$) for any $\gamma'\in \gamma +\Cc$, then $M$ is
$\gamma$-regular (resp. $\gamma$-regular at level $\ell$). One writes $ \reg_B (M):= \reg_B^0 (M)$ with
\[
 \reg_B^\ell (M):=\{ \gamma\in \G\ \vert\ M\ {\rm is}\ \gamma {\rm -regular\ at\ level}\ \ell \} .
\]
\end{defn}

It immediately follows from the definition that $\reg_B^\ell (M)$ is the maximal set
$\SS$ of elements in $\G$ such that $\SS+\Cc =\SS$ and $M$ is weakly $\gamma$-regular at level $\ell$ for any $\gamma\in \SS$.

\medskip

Before establishing the relation between \textit{weak $\gamma$-regularity at level $\ell$} and \textit{$\gamma$-regularity at level $\ell$}, we introduce a definition and some notation.

\begin{defn}\label{pers}
Let $M$ be a graded module and $\gamma \in G$. Then $B$-regularity is $\gamma$-persistent with respect to
$M$ if, for any $\eta\in G$,
$$
\eta\not\in \bigcup_{i>0} \Supp_\G(H^i_B(M))+(i-1)\gamma \quad \Rightarrow\quad \eta +\gamma \not\in \bigcup_{i>0} \Supp_\G(H^i_B(M))+(i-1)\gamma .
$$
If $B$-regularity is $\gamma$-persistent with respect to any graded module, one simply says that 
$B$-regularity is $\gamma$-persistent.
\end{defn}

We can restate Corollary \ref{pers1dir} in the following form, generalizing also Theorem 4.3 in \cite{MlS04}. 

\begin{lem}\label{pers1dir-2}
If  $B\subseteq  \sqrt{(R_\gamma )+\ann_R(M)}$, then $B$-regularity  is $\gamma$-persistent with respect to $M$. Furthermore, if $\eta\not\in \bigcup_{i>0} \Supp_\G(H^i_B(M))+(i-1)\gamma$, then
\[
M_{\eta +j\gamma }/R_{\gamma}M_{\eta +(j-1)\gamma }=H^0_B(M)_{\eta +j\gamma }/R_{\gamma}H^0_B(M)_{\eta +(j-1)\gamma }, \forall j>0.
\]
\end{lem}

Notice that it in particular implies that $B$-regularity is $\gamma$-persistent if $B\subseteq  \sqrt{(R_\gamma )}$. \\

 Let $\{ \gamma_1,\ldots ,\gamma_n\}=\{ \mu_1,\ldots ,\mu_m\}$, with $\mu_i\not= \mu_j$ for $i\not= j$ and set $B_i:=(R_{\mu_i} )$. 

\begin{thm}\label{wRtoR}
Let  $M$ be a graded $R$-module. 

\begin{enumerate}
\item If $B\subset  \sqrt{B_i+\ann_R(M)}$ for every $i\in E$, then  $B$-regularity is $\gamma$-persistent with respect to $M$, for any $\gamma \not= 0$ in the submonoid
of $\Cc$ generated by the $\mu_i$'s with $i\in E$.

\item If $M$ is weakly $\gamma$-regular at level $\ell$ and $\ell  >\cd_B( R/B_i+\ann_R(M))$  for every $i$, then $M$ is  $\gamma$-regular at level $\ell$.

\item If $M$ is weakly $\gamma$-regular at level $1$ and $B\subset \sqrt{B_i+\ann_R(M)}$ for some $i$,  then $(M/B_i M)_{\eta} =(H^0_B(M)/B_i H^0_B(M))_{\eta} $ for $\eta \in \gamma +\mu_i{\bf Z}_{>0}$.

\item If $B\subset  \sqrt{B_i+\ann_R(M)}$ for every $i$ and $M$ is $\gamma$-regular, then $(M/R_+ M)_{\gamma +\eta}=0$ for all $0\neq \eta \in \Cc$.
\end{enumerate}
\end{thm}

\begin{proof} For (1) let $\gamma :=\sum_{i\in E} \lambda_i \mu_i$. Restricting $E$ if needed, we may
assume that $\lambda_i \not= 0$ for all $i$. Then one has 
$$
B\subseteq \bigcap_{i\in E} \sqrt{B_i+\ann_R(M)}\subseteq  \sqrt{(\prod_{i\in E}(B_i)^{\lambda_i})+\ann_R(M)}\subseteq  \sqrt{(R_\gamma )+\ann_R(M)}
$$
and the assertion follows from Lemma \ref{pers1dir-2}.

We next prove (2). For $1\leq p\leq m$ let  $\FF_0^p=\FF_{0}^{(p)}=\{ 0\}$, $\FF_i^p :=\{ i\mu_p\}$,
\[
\FF_{i}^{(p)}:=\{ \gamma_{j_1}+\cdots + \gamma_{j_i}\ \vert\ j_1\leq \cdots \leq j_i\ \hbox{and}\  \gamma_{j_l}\not= \mu_p,\ \forall l\}.
\]
Applying Lemma \ref{wR} with $I:=B_p$, one gets
\[
\begin{array}{rl}
\gamma \not\in \underset{i\geq 0}{\bigcup} \Supp_\G(H^{\ell +i}_B(M))+\FF^p_{i}\ &\Leftrightarrow  
\gamma +\mu_p \not\in  \underset{i\geq 0}{\bigcup}\underset{j\geq 0}{\bigcup} (\Supp_\G(H^{(\ell +i)+j}_B(M))+\FF^p_{j+1})+\FF^p_i \ \\
&\Rightarrow  \gamma+\mu_p \not\in 
\underset{i\geq 0}{\bigcup} \Supp_\G(H^{\ell +i}_B(M))+\FF^p_{i} .\\
\end{array}
\]

For any $p$ one can write
\[
\begin{array}{rl}
\bigcup_{i\geq \ell} \Supp_\G(H^i_B(M))+\FF_{i-1}
&=\bigcup_{j\geq \ell}\bigcup_{i\geq 0}( \Supp_\G(H^{j+i}_B(M))+\FF_{i}^p)+\FF^{(p)}_{j-1}\\
\end{array}
\]

which shows that $\gamma \not\in \bigcup_{i\geq 0} \Supp_\G(H^{\ell +i}_B(M))+\FF_{i}\Rightarrow \gamma +\mu_p \not\in \bigcup_{i\geq 0} \Supp_\G(H^{\ell +i}_B(M))+\FF_{i}$ for any $p$ and concludes the proof of (2).

Statement (3) is the second part of Lemma \ref{pers1dir-2} and (4) directly follows from (3). 
\end{proof}

Next example illustrates Theorem \ref{wRtoR} in the standard multigraded case.

\begin{exmp}\label{wRtoRStMult}
Assume that $R=S[X_{ij}, 1\leq i\leq m, 0\leq j\leq r_i]$ is a finitely generated standard multigraded ring,
$B_i :=(X_{ij}, 0\leq j\leq r_i)$,
$B:=B_1 \cap\cdots \cap B_m$ and $R_+ :=B_1 +\cdots + B_m$.   Let  $M$ be a graded $R$-module. 

If $M$ is weakly $\gamma$-regular, then

\begin{enumerate}
\item $M/H^0_B(M)$ is $\gamma$-regular,

\item $(H^0_B(M)/R_+ H^0_B(M))_{\gamma'}=(M/R_+ M )_{\gamma'}$, for any $\gamma'\in \gamma +\Cc$.
\end{enumerate}
\end{exmp}

\subsection{From $B$-regularity to Betti numbers}

Next Theorem substantiate our results in Section $3$ on regularity. Together with the subsequent ones, they exhibit the importance of (weak) $\gamma$-regularity (at level $\ell$).

\begin{thm}\label{ThmLCtoTor1}
Let $M$ be a $\G$-graded $R$-module. Then
\[
 \Supp_\G(\tor^R_j(M,S)) \subset \EE_{j+1}+ \complement \reg_B (M)
\]
for all $j\not= n$, and $ \Supp_\G(\tor^R_n(M,S))\subseteq \Supp_\G(H^0_B (M))+\EE_n $.
\end{thm}

In terms of regularity, the case $j=n$ gives 
\[
\Supp_\G(\tor^R_n(M,S)) \subset \complement (\reg_B (M)+\EE_n+\EE_1 ).
\]

When $\G =\ZZ$ and the grading is standard, this reads, with the usual definition of $\reg (M)\in \ZZ$: 
\[
\reg (M)+j\geq \fin (\tor^R_j (M,S))
\]
for all $j\geq 0$. 

\begin{proof}
 If $\gamma \in \Supp_\G(\tor^R_j(M,S))$, then it follows from Theorem \ref{ThmLCtoTor} that $\gamma \in \Supp_\G(H^{\ell}_B(M))+\EE_{j+\ell}$ for some 
 $\ell$. Hence 
 $$
 \gamma -\gamma_{i_1}-\cdots -\gamma_{i_{j+\ell}} \in \Supp_\G(H^\ell_B(M))
 $$
 for some $i_1<\cdots <i_{j+\ell}$. 
 If $\ell =1$ it shows that $\mu \not\in \reg_B (M)$. If $\ell>1$, by definition, it follows that if $\mu\in \reg_B (M)$ and $t_1< \cdots < t_{\ell -1}$, then 
 $$
  \gamma -\gamma_{i_1}-\cdots -\gamma_{i_{j+\ell}} \not= \mu -\gamma_{t_1}-\cdots -\gamma_{t_{\ell -1}}
  $$
  in particular choosing $t_k:=i_{j+k+1}$ for $k>0$ one has 
   $$
  \gamma -\gamma_{i_1}-\cdots -\gamma_{i_{j+1}} \not\in \reg_B (M). 
  $$
   If $\ell =0$, by definition, it follows that if $\mu\in \reg_B (M)$
 $$
  \gamma -\gamma_{i_1}-\cdots -\gamma_{i_{j}} \not= \mu +\gamma_{t}
  $$
  for all $t$, which gives the conclusion unless $j=n$. 
  \end{proof}

The definition chosen for regularity is well fitted to the case where one has persistency 
with respect to any of the $\gamma_i$'s. Let $\Cc^*:=\Cc \setminus \{ 0\}$. We have already seen that in this case, $M$ 
is generated by elements whose degrees are not in $\mu +\Cc^*$. In particular, $M_{\mu +\Cc}$ has 
regularity $\mu +\Cc$ and is generated in degree $\mu$, for any $\mu\in \reg_B (M)$. The following result
shows that persistence in one direction improves quite much the regularity control given
by Theorem \ref{ThmLCtoTor1}.

A more precise version of Corollary \ref{CorLCtoTor} will be useful:

\begin{lem}\label{CorLCtoTor2}
Let $E=\{ i_1,\ldots ,i_t\} \subseteq \{ 1,\ldots ,m\}$ and
$$
\EE_{k}^{E}:=\{\gamma_{j_1}+\cdots+\gamma_{i_k}\ : \ i_1<\cdots<i_k,\ \gamma_{j_i}\in \{ \mu_{i_1},\ldots ,\mu_{i_t}\}, \forall i\} .
$$ 
If $B\subset  \sqrt{B_{i_1}+\cdots +B_{i_t}+\ann_R(M)}$, then for any integer $j$,
\[
\Supp_\G(\tor^R_j(M,S))\subset \bigcup_{k\geq 0}(\Supp_\G(H^{k}_B(M))+\EE_{k}^{E})+\EE_j.
\]
\end{lem}

\begin{proof}
Let $Y_E$ denote the tuple of variables whose degrees are in  $ \{ \mu_{i_1},\ldots ,\mu_{i_t}\}$, with the order induced by the one of the variables,
and $E':=\{ 1,\ldots ,m\}\setminus E$. Set $T_j^E:=\Supp_G ( \tor^R_{j}(M,R/(Y_E)))$ and $\SS^k_B :=\Supp_G (H^k_B (M))$. 
The double complex with components $\kos_p (Y_E;M)\otimes_R \kos_q (Y_{E'};M)$ whose 
totalization is isomorphic to $\k. (X;M)$ gives rise to a spectral sequence 
$$
E^1_{p,q}=\kos_p (Y_{E'}; \tor_q^R (M,R/(Y_E))) \Rightarrow \tor^R_{p+q}(M,S).
$$ 
which implies that $\Supp_\G(\tor^R_j(M,S))\subset \cup_{p+q=j}(T_p^E+\EE^{E'}_q )$.
On the other hand, the spectral sequence:
$$
_1'E^i_j = \kos_j (Y_E ; H^i_B (M)) \Rightarrow \tor^R_{j-i}(M,R/(Y_E )),
$$
implies that $T^E_p\subseteq \cup_{k\geq 0} (\SS_B^k +\EE_{p+k}^{E})$.
It follows that, 
$$
\Supp_G ( \tor^R_{j}(M,S))\subseteq \bigcup_{p+q=j}\left( \bigcup_{k\geq 0}\SS^k_B +\EE_{p+k}^{E} \right) +\EE_q^{E'} = \bigcup_{k\geq 0}(\SS_B^k +\EE_{k}^{E})+\EE_j.
$$
 \end{proof}
 
 An application of the above result gives the following interesting case:

 \begin{thm}\label{ThmLCtoTor2}
Let $M$ be a $\G$-graded $R$-module. If $B\subset  \sqrt{B_i+\ann_R(M)}$, then
\[
 \Supp_\G(\tor^R_j(M,S)) \subseteq \gamma_i +\EE_j+\complement \reg_B (M)
\]
for all $j$.
\end{thm}

\begin{proof}
Let $\CCC:=\bigcup_{k}(\SS^k_B +\EE_{k-1})$ and notice that $\SS^0_B+\EE_j \subseteq\SS^0_B+\EE_{-1}+\EE_j +\gamma_i\subseteq \CCC +\gamma_i$. Lemma \ref{CorLCtoTor2} applied to $E= \{ i\}$ implies that 
$$
\begin{array}{rl}
\Supp_G ( \tor^R_{j}(M,S))&\subseteq \left( \bigcup_{k\geq 0}\SS^k_B +k\gamma_i \right) +\EE_j
\subseteq \CCC +\EE_j +\gamma_i\\
\end{array}
$$

By definition, $\reg_B (M)+\gamma_i +\Cc \cap \CCC +\gamma_i =\emptyset$. The claim follows.
 \end{proof}

The following corollary contains the estimates for the degrees of generators of \cite[Theorem 1.3]{MlS04} (case $j=0$)
and extend it to higher syzygies ($j>0$):
\begin{cor}
Let $M$ be a $\G$-graded $R$-module. If $B\subset  \sqrt{B_i+\ann_R(M)}$ for every $i$, then
\[
 \Supp_\G(\tor^R_j(M,S)) \subseteq\bigcap_i (\gamma_i +\EE_j+\complement \reg_B (M))
\]
for any $j$. 
\end{cor}

\begin{rem}\label{remSuppFiniteR+}
Notice that if $R$ is a graded polynomial ring such that for any $\gamma\in G$, $(\gamma+\Cc)\cap (-\Cc)$ is a finite subset of $G$, then $(\gamma_i +\EE_j+\complement \reg_{R_+} (R))\cap \Cc$ is a finite set for any $i$ and $j$. Hence, if $R$ is Noetherian and $M$ is a finitely generated $R$-module, $(\gamma_i +\EE_j+\complement \reg_{R_+} (M))\cap \Supp_G(M)$ is a finite set for any $i$ and $j$. In particular, applying Theorem \ref{ThmLCtoTor2} for $B=R_+$ one gets that $\Supp_\G(\tor^R_j(M,S))$ is contained in
an explicit finite set, written in terms of regularity, for any $j$.

The finiteness of $(\gamma+\Cc)\cap (-\Cc)$ holds for instance when $\Cc$ is the monoid spanned by the degrees of the variables of the Cox ring of any product of anisotropic projective spaces.
\end{rem}

The following example evidence the relation between regularity and vanishing of Betti numbers in the probably most common situation. This result generalizes the fact that when $\G =\ZZ$ and the grading is standard, $\reg (M)+j\geq \fin (\tor^R_j (M,S))$. 

\begin{exmp}\label{exmpSuppTor}
Assume $(S,\mm,k)$ is local Noetherian, $B\subset B_i$ and let $F_\bullet$ be a minimal free $R$-resolution of a finitely generated $R$-module $M$. Then, by Theorem \ref{ThmLCtoTor2} and Lemma \ref{LemResolutions}(3)

\[
 \Supp_G(F_j \otimes_R k)\subset \bigcup_{0\leq j'\leq j}\gamma_i +\EE_{j'}+\complement \reg_B (M)=\gamma_i +\EE_{j}+\complement \reg_B (M).
\]
Also notice that
$
 \Supp_G(M \otimes_R S)\subset \gamma_i + \complement (\reg_B (M)).
$

In the case $\G=\ZZ$ and $\gamma_i=1$ for all $i$, this gives the classical inequality
\[
 \fin(F_j \otimes_R S)\leq j+1+(\reg (M)-1)=j+\reg(M)
\]
for all $j$.
\end{exmp}

\subsection{From Betti numbers to $B$-regularity}

First Corollary \ref{corSuppHi3} shows that:

\begin{prop}\label{PrTortoLC}
Assume $S$ is Noetherian, let $M$ be a finitely generated $\G$-graded $R$-module and set $T_i:=\Supp_{\G}(\tor_i^R(M,S))$. Then, for any $\ell$,
\[
 \Supp_{\G}(H^\ell_B(M)+\EE_{\ell-1} )\subset \bigcup_{i\geq j}(\Supp_{\G}(H^{\ell +i}_B(R))+\EE_{\ell-1} +T_j).
 \]
 If further $S$ is a field,
 \[
 \Supp_{\G}(H^\ell_B(M)+\EE_{\ell-1})\subset \bigcup_{i}(\Supp_{\G}(H^{\ell +i}_B(R))+\EE_{\ell-1} +T_i).
 \]
 \end{prop}

Proposition \ref{PrTortoLC} was stated requesting $S$ to be Noetherian and $M$ be a finitely generated, for the sake of simplicity. These hypotheses can omitted in the case there exists a function $f:G\to\RR$ such that $f(\gamma_i)>0$ for all $i$, and $f(\deg(\alpha))$ is bounded below for $0\neq \alpha\in M$, as in the case (2) of Lemma \ref{LemResolutions} (cf.\ \cite{CJR}). This generalization include many rings from algebraic geometry, like projective toric schemes over an arbitrary base ring.

Proposition \ref{PrTortoLC} gives the following result: 
\begin{thm}\label{ThmTortoLC}
Assume $S$ is Noetherian, let $M$ be a finitely generated  $\G$-graded $R$-module and set $T_i:=\Supp_{\G}(\tor_i^R(M,S))$. Then, 

\begin{enumerate}
\item for any $\ell \geq 1$,
\[
 \reg_B^\ell (M)\supseteq \bigcap_{i\geq j, \gamma\in T_j} \reg_B^{\ell +i} (R)+\gamma-\FF_i.
\]

\item
\[
 \reg_B (M)\supseteq \bigcap_{i\geq j, i>0, k, \gamma\in T_j} \reg_B^{i} (R)+\gamma - \gamma_k -\FF_{i-1}.
\]
\end{enumerate}

The above intersections can be restricted to $i\leq \cd_B (R)-\ell$.

If further $S$ is a field,
 \[
 \reg_B^\ell (M)\supseteq \bigcap_{i, \gamma\in T_i} \reg_B^{\ell +i} (R)+\gamma-\FF_i
 \]
 
 \end{thm}

\begin{proof}

If $\mu \not\in \reg_B^\ell (M)$,
by  Proposition \ref{PrTortoLC}, there exists $\ell'\geq \ell$ and $i\geq j$ such that 
$$
\mu \in \Supp_{\G}(H^{\ell' +i}_B(R))+\FF_{\ell'-1} +T_j.
$$
Hence, for proving (1), since $\ell'\geq \ell\geq 1$, $\FF_{\ell'-1}+\FF_i=\FF_{\ell'-1+i}$. Thus, for any $\gamma' \in \FF_i$,
$$
\mu+\gamma' \in \Supp_{\G}(H^{\ell +i}_B(R))+\FF_{i+\ell-1}  +T_j .
$$
Thus,  there exists $\gamma \in T_j$ such that $\mu+\gamma' -\gamma \in \Supp_{\G}(H^{\ell +i}_B(R))+\FF_{i+\ell-1}$, showing that $\mu  \not\in \reg_B^{\ell +i}(R)+\gamma -\gamma'$.

For (2), if $\ell >0$ or $i=0$ we proceed as above. Now assume $\ell'=0$ and $i>0$, then, there exists $\gamma \in T_j$ and $k$ such that for $\gamma_k\in \FF_1$ and all$ \gamma' \in \FF_{i-1}$,
\[
 \mu +\gamma_k -\gamma +\gamma' \in \Supp_{\G}(H^{i}_B(R))+\FF_{i-1}. \qedhere
\]
\end{proof}

\begin{exmp}
When $\G =\ZZ$ and the grading is standard, this reads with the usual definition of $\reg^\ell (M)\in \ZZ$ (notation as in Section $3$): 
$$
\reg_B^\ell (M)\leq \reg^\ell (R)+\max_i \{ \fin (\tor^R_i (M,S))-i\}  = \reg^\ell (R)+\max_i \{ b_i(M)\} .
$$
\end{exmp}


\subsection{Regularity and truncation of modules}

In this section we extend the results in \cite{MlS04} and \cite{STW06} and give sharper finite subsets of the grading group $G$ that bound the degrees of the minimal generators of a minimal free resolution of a truncation of $M$.

Lemma \ref{tronctor} below provides a multigraded variant for the bounds on the shifts in a minimal free resolution of $M_{\geq d}$ in the classical case. Here ``$M_{\geq d}$" is replaced by $M_{\Ss }$, $\Ss$ being a $\Cc$-stable subset of $G$.

As in the graded case this results take particular interest when $d\geq \reg(M)$, we apply Lemma \ref{tronctor} taking $\Ss\subseteq \reg_B(M)$, obtaining Theorem \ref{ThmLCtoTor3}. In particular, taking $\Ss=\mu+\Cc$, with $\mu\in  \reg_B(M)$ we get as corollary 
Thm.\ 5.4 in \cite{MlS04}, as well as several results in \cite{STW06} and Section 7 in \cite{MlS04} by studying higher Tor's modules.

\begin{lem}\label{tronctor}
Let $M$ be a $\G$-graded $R$-module and $\Ss \subset G$ such that $\Ss+\Cc=\Ss$. Set $N:=M_{\Ss }$. 

\begin{enumerate}

\item $\Supp_G ( \tor^R_{j}(N,S))\subseteq \EE_j +\Ss$.

\item The natural map
\[
\tor^R_{j}(N,S)_\eta \ra \tor^R_{j}(M,S)_\eta 
\]
is surjective for $\eta\in \bigcap_{\gamma\in \EE_{j}} (\gamma +\Ss) $ and an isomorphism 
for $\eta\in \bigcap_{\gamma\in \EE_{j+1}} (\gamma +\Ss)$.

\item  Let $E=\{ i_1,\ldots ,i_t\} \subseteq \{ 1,\ldots ,m\}$ and $\EE_i^E$ for $i\in\ZZ$ be as in Lemma \ref{CorLCtoTor2}. Denote by $I_E$ the ideal generated by the variables whose degrees are in  $ \{ \mu_{i_1},\ldots ,\mu_{i_t}\}$. 
The natural map,
\[
\tor^R_{j}(N,R/I_E)_\eta  \ra \tor^R_{j}(M,R/I_E)_\eta 
\]
is surjective for $\eta\in \bigcap_{\gamma\in \EE_{j}^E} (\gamma +\Ss) $ and an isomorphism 
for $\eta\in \bigcap_{\gamma\in \EE_{j+1}^E} (\gamma +\Ss)$.

\item With the hypotheses and notations of Lemma \ref{CorLCtoTor2},
\[
\Supp_\G(\tor^R_j(N,S))\subset \left( \bigcup_{k\geq 0}(\Supp_\G(H^{k}_B(M))+\EE_{k}^{E})+\EE_j\right) \cup  \left(\bigcup_{l=0}^{j} \left(\EE_{l+1}^E +\EE_{j-l}^{E'} +\complement \Ss \right)\right).
\]
\end{enumerate}
\end{lem}

\begin{proof}
For (1) notice that $\Supp_G ( \tor^R_{j}(N,S))\subseteq \Supp_G (\kos_j (X;N))=\Supp_G ( N)+\EE_j$.

For (2), consider the exact sequence $0\ra N\ra M\ra C\ra 0$ defining $C$. One has by hypothesis,
$\Supp_G (C)\subseteq G\setminus \Ss$. The exact sequence
\[
\tor^R_{j+1}(C,S)\ra \tor^R_{j}(N,S)\ra \tor^R_{j}(M,S)\ra \tor^R_{j}(C,S)
\]
shows our claim since, for any $i$,
\[
\Supp_G ( \tor^R_{i}(C,S))\subseteq \Supp_G (C)+\EE_i\subseteq G\setminus (\bigcap_{\gamma\in \EE_i}\gamma +\Ss) .
\] 
For (3) one has similarly an exact sequence,
\[
\tor^R_{j+1}(C,R/I_E)\ra \tor^R_{j}(N,R/I_E)\ra \tor^R_{j}(M,R/I_E)\ra \tor^R_{j}(C,R/I_E)
\]
and the proof follows from the inclusion $\Supp_G ( \tor^R_{i}(C,R/I_E))\subseteq \Supp_G (C)+\EE_i^E$,
that holds for any $i$.

For (4), set $T_p^E (\hbox{---}):=\Supp_\G(\tor_p^R (\hbox{---},R/(Y_E)))$ and $T_p (\hbox{---}):=\Supp_\G(\tor_p^R (\hbox{---},S))$, $\SS^k_B(M) :=\Supp_G (H^k_B (M))$. 

The spectral sequence
$
E^1_{p,q}=\kos_p (Y_{E'}; \tor_q^R (N,R/(Y_E))) \Rightarrow \tor^R_{p+q}(N,S)
$ 
implies that $T_j(N)\subset \bigcup_{p+q=j}(T_p^E (N)+\EE^{E'}_q )$.
By (3), 
\[
T_p^E (N)\subseteq T_p^E (M)\cup \complement \bigcap_{\gamma\in \EE_{p+1}^E} (\gamma +\Ss )
=T_p^E (M)\cup (\EE_{p+1}^E+\complement \Ss )
\]
The inclusion  $T^E_p(M)\subseteq \cup_{k\geq 0} (\SS_B^k (M)+\EE_{p+k}^{E})$ gives, as in the proof 
of Lemma \ref{CorLCtoTor2},
\[
\Supp_G ( \tor^R_{j}(N,S))\subseteq \left(\bigcup_{k\geq 0}(\SS_B^k (M)+\EE_{k}^{E})+\EE_j\right)
\cup \left(\bigcup_{p+q=j} (\EE^E_{p+1}+\EE^{E'}_q)+\complement \Ss \right).\qedhere
\]
\end{proof}
 
As we mentioned, Theorem \ref{ThmLCtoTor3} is a $G$-graded version of the bounding of the degrees of minimal generators of a minimal free resolution of $M_{\geq d}$, where $d\geq \reg(M)$. This is essentially obtained from Lemma \ref{tronctor} taking $\Ss\subseteq \reg_B(M)$. 

\begin{thm}\label{ThmLCtoTor3}
Let $M$ be a $\G$-graded $R$-module, $\Ss\subseteq G$ such that $\Ss +\Cc =\Ss$ and $N:=M_{\Ss}$. Then
\[
\Supp_\G(\tor^R_j(N,S))\subseteq \EE_{j+1}+\complement (\Ss \cap \reg_B (M)).
\]
 If furthermore $B\subset  \sqrt{B_i+\ann_R(M)}$, then
 \[
 \Supp_\G(\tor^R_j(N,S))\subseteq \gamma_i +\EE_{j}+\complement (\Ss \cap \reg_B (M)).
\]
\end{thm}

Also recall that, by Lemma \ref{tronctor} (1), $\Supp_\G(\tor^R_j(N,S))\subseteq \EE_{j}+\Ss$, in particular
if $\Ss\subseteq \reg_B (M)$, one has 
\[
\Supp_\G(\tor^R_j(N,S))\subseteq  (\EE_{j}+\Ss )\cap(\EE_{j+1}+\complement \Ss),
\]
and
 \[
 \Supp_\G(\tor^R_j(N,S))\subseteq  (\EE_{j}+\Ss )\cap (\gamma_i +\EE_{j}+\complement \Ss ).
\]
if  $B\subset  \sqrt{B_i+\ann_R(M)}$.

\begin{proof} By Lemma \ref{tronctor} (2), 
\[
\Supp_\G(\tor^R_j(N,S))\subseteq \Supp_\G(\tor^R_j(M,S))\cup \complement \bigcap_{\gamma\in \EE_{j+1}} (\gamma +\Ss )
\]
as $\complement \bigcap_{\gamma\in \EE_{j+1}} (\gamma +\Ss ) =\EE_{j+1} +\complement \Ss$ the first inclusion follows from Theorem \ref{ThmLCtoTor1}. 

The second inclusion follows from Lemma \ref{tronctor} (4), along the lines of the proof of 
Theorem \ref{ThmLCtoTor2}, observing that whenever $E=\{ i\}$, $\EE^E_{l+1}+\EE_{j-l}^{E'}\subseteq \gamma_i +\EE_j$.
\end{proof}

We use again Example \ref{exmpSuppTor} in order to show the relation between regularity and vanishing of Betti numbers in the probably most common situation for a truncation of $M$. This result generalizes the fact that when $\G =\ZZ$ and the grading is standard, $\reg (M_{\geq d})+j\geq \fin (\tor^R_j (M_{\geq d},S))$ if $d\geq \reg(M)$. 

\begin{exmp}\label{exmpSuppTorBis}
Consider Example \ref{exmpSuppTor} again and take any $\mu\in G$, and $\Ss=\mu+\Cc$.

Lemma \ref{tronctor} says that if $N:=M_{\Ss}$. 

\begin{enumerate}
\item $\Supp_G ( \tor^R_{j}(N,S))\subseteq \mu+\EE_j +\Cc$. In particular, if the grading is standard, since $\reg(M)\leq \mu$ in $\ZZ$, this give $\Supp_\ZZ(\tor^R_{j}(N,S)) \geq \mu+ j$.

\item If $\eta\in \mu+\bigcap_{\gamma\in \EE_{j+1}} (\gamma +\Cc)$ then, 
\[
\tor^R_{j}(N,S)_\eta \cong \tor^R_{j}(M,S)_\eta 
\]
\end{enumerate}

With the standard grading we get,
\[
\tor^R_{j}(N,S)_{\geq \mu+j+1} \cong \tor^R_{j}(M,S)_{\geq \mu+j+1} 
\]

Theorem \ref{ThmLCtoTor3} with $\mu\in \reg_B(M)$ says that
$\Supp_\G(\tor^R_j(N,S))\subseteq \mu+\EE_{j+1}+\complement \Cc.$

Again taking the standard grading we obtain,
\[
\fin(\tor^R_j(N,S))\leq \mu+j+1+\complement \ZZ_{\geq 0}=\mu+j.
\]
This, together with point (1) above gives $\Supp_\ZZ(\tor^R_{j}(N,S)) = \mu+ j$ that in turns shows the fact that in the classical case $M_{\geq \mu}$ admits a $\mu$-linear resolution

\end{exmp}

We now compare the cohomology modules of $M$ with the cohomology modules of the truncations of $M$.

\begin{lem}\label{regtrunc}
Let $M$ be a $\G$-graded $R$-module, let $\Gamma$ be the submonoid of $\Cc$ generated by the degrees of elements of $B$, and take $\Ss \subset G$ such that $\Ss+\Cc=\Ss$. Set $N:=M_{\Ss }$, $R_\Gamma$ the subring of $R$ of elements of degree in $\Gamma$, and $\ZZ \Gamma \subseteq G$ be the subgroup generated by $\Gamma$. If the $R_\Gamma$-module $M_{\Ss +\ZZ\Gamma}/M_\Ss$ is $B$-torsion, then:

\begin{enumerate}

\item $H^i_B(M)_\gamma =H^i_B(M_{\gamma +\ZZ \Gamma})_\gamma$, for every $i$ and $\gamma\in G$.

\item $H^i_B(N)_\gamma =0$, for all $i$, if $\gamma\not\in \Ss +\ZZ\Gamma$.

\item $H^i_B(N)_\gamma =H^i_B(M)_\gamma$ for $i\geq 2$ and $\gamma\in \Ss +\ZZ\Gamma$.

\item For $\gamma\in \Ss +\ZZ\Gamma$, 
\[
 H^1_B(N)_\gamma \neq 0 \iff H^1_B(M)_\gamma\neq 0 \mbox{ or }\gamma\in \Supp_G(M/H^0_B(M))\setminus \Ss.
\]

\item $\Supp_G(H_B^0(N))=\Supp_G(H_B^0(M))\cap \Ss$, and $H_B^0(N)_\gamma=H_B^0(M)_\gamma$ if $H_B^0(N)_\gamma \not= 0$.
\end{enumerate}
\end{lem}

\begin{proof}
Set $M':=M_{\gamma +\ZZ\Gamma }$.

For (1) notice that the complexes of $R_\Gamma$-modules $\check \Cc^\bullet_B (M)_\gamma$ and $\check \Cc^\bullet_B (M')_\gamma$ coincide. 

Statement (2) holds since the complex of $R_\Gamma$-modules $\check \Cc^\bullet_B (N)_\gamma$
is the zero complex  whenever $\gamma\not\in \Ss +\ZZ\Gamma$.

The short exact sequence $0\to N\to M'\to C\to 0$ induces a long exact sequence 
on cohomology that proves (3) using (1), since $H^i_B (C)=0$ for $i>0$ as  $C$ is $B$-torsion. It also 
provides an exact sequence
$$
0\to  H^0_B (N)\to H^0_B (M')\to M'/N \to H^1_B (N)\to H^1_B (M')\to 0
$$

This sequence shows (5) and proves that $H^i_B (N)_\gamma =H^i_B (M')_\gamma$  for all $i$ if $\gamma
\in \Ss$. For $\gamma\in ( \Ss +\ZZ\Gamma)\setminus \Ss$, it gives the exact sequence 
$$
0\to H^0_B (M)_\gamma \to M_\gamma \to H^1_B (N)_\gamma \to H^1_B (M)_\gamma \to 0
$$
which completes the proof.
\end{proof}

Notice that when $G=\ZZ^r$ and the grading is the classical multigrading, then $\Cc=\ZZ_{\geq 0}^r$, and $\Gamma=(1,\hdots,1)\cdot \ZZ_{\geq 0}^r$. Thus, $\ZZ\Gamma=G$, and in particular $\Ss+\ZZ\Gamma=G$.
\begin{rem}
The equality $H^i_B (N)_\gamma =H^i_B (M)_\gamma$,  for $\gamma\in \Ss$ and all $i$, that we have proved above, is sufficient for most of the applications.
\end{rem}

\begin{prop}\label{propRegEll}
Let $M$ be a $\G$-graded $R$-module, let $\Gamma$ be the submonoid of $\Cc$ generated by the degrees of elements of $B$, and take $\Ss \subset G$ such that $\Ss+\Cc=\Ss$.  If $M_{\Ss +\ZZ\Gamma}/M_\Ss$ is $B$-torsion, for instance if $B\subseteq B_i$ for all $i$, then:

\begin{enumerate}
\item $ \reg^\ell_B (M_\Ss )=\reg^\ell_B (M)\cap (\Ss +\ZZ\Gamma )$, for $\ell \geq 2$.

\item $ \reg^\ell_B (M_\Ss )\cap \Ss =\reg^\ell_B (M)\cap \Ss $ for all for $\ell$.
\end{enumerate}

\end{prop}

\begin{proof}
Statement (1) follows directly from Lemma \ref{regtrunc} (2) and (3). 

Statement (2) follows from \ref{regtrunc} (3) for $\ell \geq 2$ and from \ref{regtrunc} (3) and (4) for $\ell = 1$. 

As $H^0_B(N)_\gamma \subset H^0_B(M)_\gamma$ for all $\gamma \in  G$ and 
$ \reg^1_B (M_\Ss )\cap \Ss =\reg^1_B (M)\cap \Ss $, one has 
 $ \reg_B (M_\Ss )\cap \Ss \supset \reg_B (M)\cap \Ss $, and if this inclusion is strict, there exists $\gamma\in  \reg_B (M_\Ss )\cap \reg^1_B (M)\cap\Ss$ such that $M$ is not weakly $\gamma$-regular. This implies that $\gamma\in \Supp_G(H^0_B(M))+\EE_{-1}$, hence $\gamma +\gamma_i\in \Supp_G(H^0_B(M))$ for some $i$. Since $\gamma+\gamma_i\in \Ss$ and $\gamma +\gamma_i\notin \Supp_G(H^0_B(M_\Ss))$ this contradicts \ref{regtrunc} (5).
\end{proof}

\begin{exmp}
When $G=\ZZ$ and the grading is standard, if $\Ss=\ZZ_{\geq d}$, Proposition \ref{propRegEll} says the following: $\max\{\reg(M_{\geq d}),d\} =\max\{\reg(M),d\}$, showing that 
$\reg(M_{\geq d}) =\max\{\reg(M),d\}$ unless $M_{\geq d}=0$.
\end{exmp}

\begin{rem} In the case  $M_{\Ss +\ZZ\Gamma}/M_\Ss$ is $B$-torsion, Theorem \ref{ThmLCtoTor3} follows from Theorems \ref{ThmLCtoTor1} and \ref{ThmLCtoTor2} using Proposition \ref{propRegEll}(2).
\end{rem}


\subsection{Example: a hypersurface in $\PP^1\times \PP^1$.}

Next example illustrates in a very simple case that local cohomology vanishing of a principal ideal depends not only on the degree but on the generator, contrary to the classical theory. For simplicity we treat the case of a form of bidegree $(1,1)$, we show that the vanishing of local cohomology depends on the factorization of the form. The same kind of phenomenon occur for any bidegree. This is not the case in $\PP^r\times \PP^s$: the support depends in this case at least on the heights of the content of the polynomial with respect to the two sets of variables.

In this example we show that the bound on Tor is sharp in the sense that an $R$-module with regularity $\ZZ^2_{\geq 0}=\reg_B(R)$ can have a first syzygy of bidegree $(1,1)$. It also shows that two modules with the same resolution may have different regularity al level $2$.

\medskip

Let $k$ be a field, and $R=k[X_1,X_2, Y_1,Y_2]$ with its standard bigrading,
$B_1 :=(X_1,X_2)$ and $B_2 :=(Y_1,Y_2)$
$B:=B_1 \cap B_2$ and $R_+ :=B_1 + B_2$.   Let  $F\neq 0$ be a bi-homogeneous element of $R$ of bidegree $(1,1)\in \ZZ^2$. 

The exact sequence $0\to R(-1,-1)\nto{\times F}{\lto} R \to R/F\to 0$ gives rise to the long exact sequence
\begin{equation*}\label{l.e.s.}
 \xymatrix@R-20pt@C-15pt {
  0\to H^1_B(R/F)\ar[r] 
  & H^2_B(R(-1,-1))\ar[r]^{\qquad \varphi} 
  & H^2_B(R)\ar[r] 
  & H^2_B(R/F)\ar `[r] `[d] '[dll] *{} `[ddll] `[ddll] [ddl]
  & \\
  &&& \\
  {}
  &
  & H^3_B(R(-1,-1))\ar[r]^{\qquad \psi}  
  & H^3_B(R)\ar[r] 
  & H^3_B(R/F)\ar[r] 
  & 0, 
}
\end{equation*}
where $\varphi$ and $\psi$ are multiplication by $F$.
The commutative square 
\begin{equation*}\label{square}
 \xymatrix {
  H^3_B(R(-1,-1))\ar[d]^{\simeq}\ar[r]^{\qquad \psi} & H^3_B(R)\ar[d]^{\simeq}\\
  H^4_{R_+}(R(-1,-1))\ar[r]^{\qquad \times F} & H^4_{R_+}(R)
}
\end{equation*}
identifies $K:=\ker(\psi)$ with $H^3_{R_+}(R/F)$,
 which is the graded $k$-dual of $\omega_{R/F}=(R/F)(-1,-1)$. Hence, $\Supp_{\ZZ^2}(K)=\ZZ_{<0}^2$.

One further notice that the map $\varphi_{(a,b)}$ is not injective for $a=-1$ and $b\geq 1$ and $b=-1$ and $a\geq 1$ as this map goes from a non-zero source to zero. 
It follows that  $ \Supp_{\ZZ^2}(H^1_B(R/F)) \supset (\{-1\}\times \ZZ_{>0})\cup  (\ZZ_{>0}\times \{-1\})$.

As $ \Supp_{\ZZ^2}(H^2_B(R/F))\subset \Supp_{\ZZ^2}(K) \cup \Supp_{\ZZ^2}(H^2_B(R))$ one deduces that 
\[
\reg_B(R/F)= \ZZ^2_{\geq 0}.
\]

\begin{rem} Notice that $\reg_B(R/F)$ does not depend on $F$ even though the support of $H^1_B(R/F)$ and $H^2_B(R/F)$  do depend on $F$, and $\reg^2_B(R/F)$ as well.

Indeed, the supports of $H^1_B(R/F)\simeq \ker(\varphi)$ and $C:=\coker(\varphi)$ depend upon the reducibility of $F$.

The following picture shows the support of the two cohomology modules in the two cases. For $H^1_B(R/F)$: $\blacktriangle$ indicates a bidegree that is in the support of $H^1_B(R/F)$ independently of $F\neq 0$ (of bidegree $(1,1)$), and $\triangle$ indicates a bidegree where $H^1_B(R/F)\neq 0$ iff $F$ is a product of two linear forms. The non-marked spots are never in the support of  $H^1_B(R/F)$.

For $H^2_B(R/F)$: $\bullet$ indicates a bidegree in the support of $K$ (which is independent of $F\neq 0$), $\blacksquare$ stands for a bidegree where $C\neq 0$ independently of $F\neq 0$,  and $\square$ indicates a bidegree where $C\neq 0$ iff $F$ is a product of two linear forms. The non-marked spots on the upper-right part are never in the support of  $H^2_B(R/F)$.

\begin{center}
\includegraphics{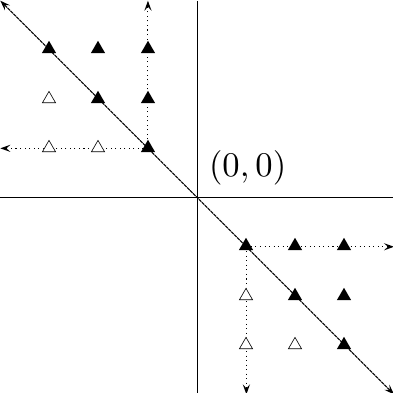}
\qquad
\includegraphics{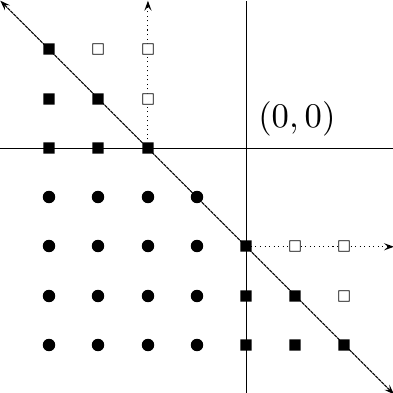}

$\Supp_{\ZZ^2}(H^1_B(R/F))$\hspace{2cm} $\Supp_{\ZZ^2}(H^2_B(R/F))$
\end{center}
\end{rem}

The computation of the regions marked with squares comes from the explicit calculation of $H^1_B(R/F)$ and $C$. After linear change of coordinates one is reduced to treat the two cases: $F_1= X_1Y_1$ and $F_2=X_1Y_1+X_2Y_2$.

Notice that the source and target of $\varphi_{(a,b)}$ vanish unless $a\leq -1$ and $b\geq 0$ or $b\leq -1$ and $a\geq 0$. We treat the first case, the second can be treated the same way by changing the roles of $X$ and $Y$.

Now, $\varphi_{(-1,*)}=0$ and when $a\leq -2$, $\varphi_{(a,*)} : k[Y_1,Y_2](-1)^{-a}\to k[Y_1,Y_2]^{-a-1}$ is given by one of the two following matrices
\[
\begin{pmatrix}
Y_1 & 0 & \cdots & 0& 0 \\
0 & Y_1 & \ddots & \vdots & \vdots\\
\vdots & \ddots& \ddots & 0& \vdots\\
0 & \cdots& 0& Y_1&0
\end{pmatrix}
\quad \mbox{and}\quad
\begin{pmatrix}
Y_1 & Y_2 & 0& \cdots & 0 \\
0 & Y_1 & Y_2 & \ddots & \vdots\\
\vdots & \ddots& \ddots & \ddots& 0\\
0 & \cdots& 0& Y_1&Y_2
\end{pmatrix},
\]
the first for $F_1$ and the second for $F_2$. It follows that the kernel of $\varphi_{(a,*)}$ is $k[Y_1,Y_2](-1)$ in the first case, and $k[Y_1,Y_2](a)$ in the second case. The cokernel of $\varphi_{(a,*)}$ is $k[Y_2]^{-a-1}$ in the first case and is $(Y_1,Y_2)$-primary generated in degree $0$ and of regularity $-a+2$ in the second case.

A similar computation shows that for a form $F$ of arbitrary bidegree $(d,e)$ the regularity is $(d-1,e-1)+\reg_B(R)$, and the support of local cohomology depends upon the existence of factors of $F$ in $k[X_1,X_2]$, or $k[Y_1,Y_2]$, or both.

\medskip

Applying Theorem \ref{ThmLCtoTor2} to an $R$-module $M$ with regularity and support in $\ZZ^2_{\geq 0}$, one obtains that $\Supp_\G(\tor^R_1(M,k))$ is in the intersection 
\[
((1,0) +\EE_1+\complement \ZZ^2_{\geq 0})\cap((0,1) +\EE_1+\complement \ZZ^2_{\geq 0})\cap (\ZZ^2_{\geq 0}+\EE_1),
\]
which is equal to $(0\times \ZZ_{>0}) \cup (\ZZ_{>0}\times  0) \cup \{(1,1)\}$.

Taking $M=R/F$ as above one sees that indeed $(1,1)\in \Supp_{\ZZ^2}(\tor^R_1(R/F,k))$ may occur.

\medskip

For concluding this section we leave the following open question:

\begin{ques}\label{questionReg}
Is there a ring $R$, a finitely generated ideal $B$ and two modules $M$ and $N$ satisfying: $\Supp_\G(\tor^R_i(M,k))=\Supp_\G(\tor^R_i(N,k))$, and $\reg_B(M)\neq \reg_B(N)$?
\end{ques}

This example given provides a positive answer for $\reg^\ell_B(M)\neq \reg^\ell_B(N)$ when $\ell=2$, and we expect the answer to be positive for $\ell=0$.

\subsection{Hilbert functions and regularity for standard multigraded polynomial rings}\label{secHF}

The aim of this part is to recall the application of Castelnuovo-Mumford regularity to the study of Hilbert functions. 
Lemma \ref{lemClassCC} provides a short proof of Grothendieck-Serre formula for standard multigraded polynomial rings. This lemma shows in particular that if the function $F_R$ (given in \eqref{eqFM}) of a multigraded polynomial ring $R$ belongs to a class closed under shifts and addition then so does $F_M$ for any finitely generated $R$-module $M$.
 
Let $R$ be a Noetherian polynomial ring over a field $k$, graded by an abelian group $G$ and $B$ be a non trivial graded ideal.
Assume that $H^i_B(R)_\mu$ is a finite dimensional $k$-vector space for any $\mu\in G$. 

For a finitely generated graded $R$-module $M$ set $[M](\mu ):=\dim_k (M_\mu )$ and
\begin{equation}\label{eqFM}
F_M (\mu ):=[M](\mu )-\sum_i (-1)^i [H_B^i(M)](\mu ).
\end{equation}
It follows from the proof of Lemma \ref{lemClassCC} below that $[H_B^i(M)](\mu )$ is finite for any $i$ and $\mu$. Recall that in the standard graded situation, $F_M$ is a polynomial function, called the Hilbert polynomial of $M$.

\begin{lem} \label{lemClassCC}
Let $\CC$ be the smallest set of numerical functions from $G$ to $\ZZ$ containing $F_R$ such that for any $F,G\in \CC$ and $\gamma\in G$, the
function $F+G$, $-F$ and $F\{ \gamma \}: g\mapsto F(\gamma +g)$ are in $\CC$. 

Then $\CC$ coincides with the set of functions of the form $\sum_{i=0}^{s} (-1)^iF_{M_{i}}$ with $s\in \ZZ_{\geq 0}$ and the $M_i$'s in the category of finitely generated graded $R$-modules.
\end{lem}

\begin{proof} First notice that any function in $\CC$ can be written in the form $\sum_{i=0}^{s} (-1)^iF_{M_{i}}$, with $M_i=R[\gamma_i]$ for some $i$.
On the other hand if $F_\bullet$ is a graded finite free $R$-resolution of $M$,
$[M]=\sum_j (-1)^j  [F_j]$ and the spectral sequence $H^i_B (F_j) \Rightarrow H^{i-j}_B (M)$
shows that  $H_B^i(M)_\mu$ is a finite dimensional vector space for any $\mu$ and that
$$
\sum_{i,j}(-1)^{i-j} [H_B^i(F_j)] =\sum_\ell (-1)^\ell [H_B^\ell (M)] .
$$
Since $F_j=\oplus_{q\in E_j} R[\gamma_{j,q}]$, it follows that 
\[
 F_M = \sum_j (-1)^j  [F_j] +\sum_{i,j}(-1)^{i-j} [H_B^i(F_j)] =  \sum_j (-1)^j \sum_{q\in E_j}F_{R}\{ \gamma_{j,q}\} \in \CC. \qedhere
\]
\end{proof}

Lemma \ref{lemClassCC} shows in particular that if $F_R$ is a numerical polynomial (resp.\ quasi-polynomial) of degree $d$, then the class $\CC$ is contained in the class of numerical polynomials (resp.\ quasi-polynomials) of degree at most $d$.

In the case $R:=k[X_{i,j}\ \vert\ 1\leq i\leq n, 0\leq j\leq r_i]$ is a standard $\ZZ^n$-graded polynomial ring over a field $k$, $\deg(X_{i,j})=e_i$, let $B:=\cap_{i=1}^{n} (X_{i,j},\  0\leq j\leq r_i) $. 
The following result due to G.\ Colom\'e i Nin generalizes \cite[Thm.\ 2.4]{JV02}, which considers the case $n=2$.

\begin{prop}\label{HF} \cite[Prop.\ 2.4.3]{ColTesis} Let $R$ be a standard $\ZZ^n$-graded Noetherian polynomial ring over $k$, $B$ be defined as above, and $M$ be a finitely generated graded $R$-module. For any $\mu \in \ZZ^n$, $H_B^i(M)_\mu$ is a finite dimensional vector space and there exists a numerical polynomial $P_M$ such that 
\[
 [M] (\mu ) =P_M (\mu )+\sum_i (-1)^i  [H_B^i(M)] (\mu ).
\]
\end{prop}

From Proposition \ref{HF} and the definition of $\reg_B$ we conclude that 

\begin{cor}
 Under the hypothesis of Proposition \ref{HF}, there exists a numerical polynomial $P_M$ such that 
\[
 [M] (\mu ) =P_M (\mu), \ \forall \mu\in \reg_B (M).
\]
\end{cor}

An alternative proof of Proposition \ref{HF} is given by Lemma \ref{lemClassCC} and the following result
\begin{lem}
With the above notations,
\[
 F_R(a_1,\hdots,a_s)=\prod_{1\leq i\leq s}\binom{r_i+a_i}{r_i}.
\]
and $\CC$ is the set of numerical polynomials of multidegree $\leq (r_1,\ldots ,r_n)$.
\end{lem}
\begin{proof}
The computation of $F_R$ follows from the explicit description of $H_B^\ell (R)$ as a direct sum of local cohomology with support on ideals generated by variables (see the proof of Lemma \ref{lemLCMonomial}, or \cite[Lem.\ 6.4.7]{BotPhD} for  more details), and the fact that  
\[
\binom{r+(-a-r-1)}{r}=(-1)^{r}\binom{r+a}{r}.
\]
The second claim comes from the first and \cite[Thm.\ 2.1.7]{Rob98}.
\end{proof}

Proposition \ref{HF} can be extended in several direction, for instance to a product of anisotropic projectives or when set set of degrees of variables has linearly independent elements (see \cite[Prop.\ 2.4.3]{ColTesis}). Using a result of Sturmfels on vector partition function  \cite{Stu94}, Maclagan and Smith also treated the case of a smooth toric variety in \cite[Lem.\ 2.8 and Prop.\ 2.14]{MS05}.


\def\cprime{$'$}


\end{document}